\documentclass[
a4paper,
11pt,
twoside,
]{article}
\usepackage[utf8]{inputenc}
\usepackage[T1]{fontenc}
\usepackage[sc]{mathpazo}
\usepackage[english]{babel}
\usepackage{geometry}
\usepackage{amsmath,amsfonts,amssymb,amsthm}
\usepackage[final]{hyperref}
\usepackage{graphicx}
\usepackage{mathtools}
\usepackage[dvipsnames]{xcolor}
\usepackage{pdflscape}
\usepackage{etoolbox}
\usepackage{fancyhdr}
\usepackage{lastpage}
\usepackage{ifdraft}
\usepackage{hyphenat}
\ifdraft{\usepackage{showlabels}}{}
\usepackage{enumitem}
\usepackage{siunitx}
\usepackage[mathlines,pagewise]{lineno}

\newcommand{\dd}{\mathop{}\!\mathrm{d}}
\newcommand{\R}{\mathbb{R}}

\theoremstyle{plain}
\newtheorem{theorem}{Theorem}[section]

\newtheorem{proposition}[theorem]{Proposition}

\newtheorem{remark}[theorem]{Remark}

\numberwithin{equation}{section}
\numberwithin{table}{section}
\numberwithin{figure}{section}

\fancypagestyle{mypagestyle}{%
  \fancyhf{}%
  \fancyhead[LO,RE]{\scriptsize Computing periodic solutions of RFDEs/REs}%
  \fancyhead[RO,LE]{\scriptsize A. And\`o and D. Breda}%
  \fancyfoot[LE,RO]{\scriptsize \thepage\,/\,\pageref*{LastPage}}%
}
\title{Piecewise orthogonal collocation for computing periodic solutions of coupled delay equations}
\author{
Alessia And\`o$^{1,2,3}$ and Dimitri Breda$^{2,4}$\\[.5em]
\small $^{1}$Area of Mathematics, Gran Sasso Science Institute\\[-.2em]
\small via Francesco Crispi 7, 67100 L'Aquila, Italy\\[.5em]
\small $^{2}$CDLab -- Computational Dynamics Laboratory\\[-.2em]
\small Department of Mathematics, Computer Science and Physics -- University of Udine\\[-.2em]
\small via delle scienze 206, 33100 Udine, Italy\\[.5em]
\small $^{3}$\texttt{alessia.ando@gssi.it}\\[.5em]
\small $^{4}$\texttt{dimitri.breda@uniud.it}
}
\date{\today}

\begin{document}
\clearpage
\maketitle
\thispagestyle{empty}

\begin{abstract}
We extend the piecewise orthogonal collocation method to computing periodic solutions of coupled renewal and delay differential equations. Through a rigorous error analysis, we prove convergence of the relevant finite-element method and provide a theoretical estimate of the error. We conclude with some numerical experiments to further support the theoretical results.\end{abstract}




\section{Introduction}\label{s_introduction}

Phenomena such as incubation periods in infectious diseases and maturation in population dynamics clearly indicate that natural processes do not happen instantaneously. In other words, there are delays between actions and their consequences, which is what leads to the introduction of delay equations, now widely used in various areas of mathematical modeling such as pharmacokinetics and behavioral epidemiology \cite{md86,dms07,krz19}.
%

In all the aforementioned contexts, as well as in many others \cite{ern09}, the focus of the numerical investigations is not limited to the time integration of initial-value problems associated to the equations, but rather extends to the analysis of their long-term dynamics. This includes as measure targets the detection and computation of equilibria and periodic solutions, possibly in view of performing a bifurcation analysis.
As far as delay equations are concerned, the piecewise orthogonal collocation method for computing periodic solutions was first introduced in \cite{elir00} for Retarded Functional Differential Equations (RFDEs), although a theoretical convergence analysis was obtained much later, in \cite{ab20mas}, following the approach in \cite{mas15NM} devoted to abstract Boundary Value Problems (BVPs). The extension of the method to Renewal Equations (REs) has been first applied in \cite{bdls16}, later formally described in \cite{aa20} and \cite{ab19} and finally proved to be convergent in \cite{ab22RE} and \cite{adena21}.
%

In population dynamics, RFDEs and REs frequently appear in the form of coupled systems (\cite{dgmnr10} is a prominent example). This is what motivates the need to extend the piecewise orthogonal collocation method to such systems and, correspondingly, prove its convergence theoretically, which constitutes indeed the purpose of the present paper. Despite the possibility to exploit and combine the results already obtained for RFDEs and REs separately, new obstacles arise in the coupling from the fact that RFDEs and REs are typically defined on spaces having different regularity. Moreover, a major role is played by the structural difference between the two kinds of equations: while the right-hand side of a RFDE prescribes the current-time value of the
derivative of the unknown function, that of a RE directly prescribes the current-time value of the unknown function.
%

The contribution of the present work is twofold. On the one hand,
to the best of our knowledge, this is the first numerical tool to
compute periodic solutions of coupled delay equations properly as
solutions of BVPs. As such, we hope that it will help in fostering the
study of physiologically-structured populations\footnote{As far as
initial value problems for the latter class are concerned, let us
mention the {\it escalator boxcar train} approach \cite{ebt88} and also the works \cite{aalm22fin,aalm22inf}.}. On the other hand, as for the convergence of the method, the main result we could prove confirms the reasonable expectation of an error decaying as $O(L^{-m})$ for a piecewise polynomial of degree $m$ defined on $L$ mesh intervals and under sufficient regularity of the right-hand sides of the coupled system {(which are not restrictive for common models)}. In particular, this result refers to the {\it finite element} approach, i.e., $m$ fixed while $L\to\infty$.

Details on the equations of interest in this paper will be provided in the following Subsection \ref{s_introre}, together with a remark on the standard formulation of the problem of computing periodic solutions as a BVP. Then, the description of the piecewise orthogonal collocation method for coupled systems will be given in detail in Section \ref{s_piecewise}. In Section \ref{s_convergence} we carry out the actual convergence analysis, while emphasizing its connection to the abstract approach in \cite{mas15NM}, as well as the new aspects to take into account with respect to the separate cases of only RFDEs \cite{ab20mas} or only REs \cite{ab22RE}. The most technical (yet demanding) part of the proof is moved to \ref{s_appendix}. Eventually, Section \ref{s_results} shows some experimental results confirming the theoretical findings. Python demos are freely available at \url{http://cdlab.uniud.it/software}.

\subsection{Coupled systems and boundary value problems}
\label{s_introre}
In its most general form, a coupled RE/RFDE system can be written as
\begin{equation}\label{eq:coupled}
\left\{\setlength\arraycolsep{0.1em}\begin{array}{rcll}
x(t) &=& F(x_t,y_t), \\[1mm]
y'(t) &=& G(x_t,y_t),
\end{array}
\right.
\end{equation}
where, for a positive integer $d$, $F,G:\mathtt{X}\times\mathtt{Y}\rightarrow\mathbb{R}^{d}$ are autonomous, in general nonlinear and the \emph{state spaces} $\mathtt{X},\mathtt{Y}$ are sets of functions from $[-\tau,0]$ to $\mathbb{R}^{d}$ for some $\tau>0$, called the \emph{delay}\footnote{We assume for simplicity to work with the same number $d$ of REs and RFDEs. A generalization to $d_X$ REs and $d_Y$ RFDEs is straightforward and we prefer not to overload the notation.}. The \emph{state} at time $t$ of the dynamical system on  $\mathtt{X}\times\mathtt{Y}$ generated by \eqref{eq:coupled} is $(x_{t},y_{t})$, with $x_{t}$ defined as $x_{t}(\theta):=x(t+\theta),\,\theta\in[-\tau,0]$, and $y_{t}$ defined similarly\footnote{For the well-posedness of the initial value problem for \eqref{eq:coupled} see \cite{dgg07}.}.
In particular, we develop our analysis for $\mathtt{X}:=B^{\infty}([-\tau, 0],\mathbb{R}^d)$, where $B^{\infty}$ denotes bounded and measurable functions, and $\mathtt{Y}:=B^{1,\infty}([-\tau, 0],\mathbb{R}^d)$, where $B^{1,\infty}$ denotes continuous functions having derivatives in $B^{\infty}$. The elements of $B^{\infty}$ are considered as functions, not as classes of functions that are equal almost everywhere. Such choice, instead of the classical $L^1([-\tau,0],\mathbb{R}^d)$ \cite{diekmann95}, is justified by the need of evaluating the functions pointwise in order to deal with collocation.
A periodic solution of \eqref{eq:coupled} with period $\omega>0$  \footnote{We use the letter $\omega$ to indicate the period, following the notation of \cite{hvl93}.}, if there is any, can be obtained by solving a BVP where the solutions are considered over just one period and the periodicity is imposed to the solution values at the extrema of the period.
Note that this requires to evaluate $x$ and $y$ at points that fall off the interval $[0,\omega]$, due to the delay. In order to deal with this issue, one can exploit the periodicity to bring back the evaluation to the domain $[0,\omega]$, which, assuming $\omega\geq \tau$\footnote{A solution with period $\omega$ is also a solution with period $k\omega$ for any positive integer $k$. Moreover, in the hyperbolic case, the stability of the relevant orbit is preserved and $t+\theta\geq -\omega$ holds for all $t\in[0,\omega]$.\label{tauleqomega}}, means defining the function $\overline{x}_t\in \mathtt{X}$ as

\begin{equation}\label{perstate}
\overline{x}_t(\theta):=\left\{\setlength\arraycolsep{0.1em}\begin{array}{ll}
x(t+\theta),  &\quad t+\theta\in[0,\omega], \\[1mm]
x(t+\theta+\omega), &\quad t+\theta\in[-\omega,0), \\[1mm]
\end{array}
\right.
\end{equation}
and $\overline{y}_t\in\mathtt{Y}$ similarly. The periodic BVP can then be formulated as
\begin{equation}\label{barBVPcoupled}
\left\{\setlength\arraycolsep{0.1em}\begin{array}{rcll}
x(t) &=& F(\overline{x}_t,\overline{y}_t),&\quad t\in[0,\omega], \\[1mm]
y'(t) &=& G(\overline{x}_t,\overline{y}_t),&\quad t\in[0,\omega], \\[1mm]
x(0) &=& x(\omega)&\\[1mm]
y(0) &=& y(\omega)&\\[1mm]
p(x,y)&=&0,&
\end{array}
\right.
\end{equation} 
where $p$ is a scalar (usually linear) function defining a so-called \emph{phase condition}, necessary to remove translational invariance \cite{doe07}. This is the most natural BVP formulation of the problem following the case of RFDEs, e.g., \cite{bad85,bkw06,bz84,elir00,liu94,mas15I,rt74}, as well as the most convenient to consider when developing the numerical method we propose here. However, we will introduce in Section \ref{s_convergence} a slightly different, yet equivalent, formulation in view of the convergence analysis based on \cite{mas15NM}.

\bigskip
\noindent In realistic models, such as those describing structured populations \cite{md86}, $F$ and $G$ have usually the form
\begin{equation}\label{hpcoupled1}
F(\varphi,\psi)=\tilde F(\tilde K(\varphi,\psi),\psi), \qquad G(\varphi,\psi)=\tilde G(\tilde H(\varphi,\psi),\psi)
\end{equation}
for some functions $\tilde F,\tilde G:\mathbb{R}^d\times\mathtt{Y}\to\mathbb{R}^d$, where, in turn,
\begin{equation}\label{hpcoupled2}
\tilde K(\varphi,\psi)=\int_{-\tau}^0K(\sigma,\varphi(\sigma),\psi(\sigma))\dd\sigma,\quad \tilde H(\varphi,\psi)=\int_{-\tau}^0H(\sigma,\varphi(\sigma),\psi(\sigma))\dd\sigma
\end{equation}
for some integration kernels $K,H:[-\tau,0]\times\mathbb{R}^{2d}\to\mathbb{R}^d$.

\section{Piecewise orthogonal collocation}
\label{s_piecewise}
This section describes the numerical method used to compute periodic solutions of \eqref{eq:coupled}, starting from general right-hand sides $F,G$ as in \eqref{hpcoupled1} which, for the moment, are assumed to be computable without resorting to any further numerical approximation.

\bigskip
\noindent Since the period $\omega$ is unknown, it is numerically convenient (see, e.g., \cite{elir00}) to reformulate \eqref{barBVPcoupled} through the map $s_{\omega}:\mathbb{R}\to\mathbb{R}$ defined by $s_{\omega}(t):=t/\omega$. \eqref{barBVPcoupled} is thus equivalent to
\begin{equation}\label{rescaledBVPcoupled}
\left\{\setlength\arraycolsep{0.1em}\begin{array}{rcll}
x(t) &=& F(\overline{x}_t\circ s_{\omega},\overline{y}_t\circ s_{\omega}),&\quad t\in[0,1], \\[1mm]
y'(t) &=& \omega G(\overline{x}_t\circ s_{\omega},\overline{y}_t\circ s_{\omega}),&\quad t\in[0,1], \\[1mm]
x(0) &=& x(1)&\\[1mm]
y(0) &=& y(1)&\\[1mm]
p(x,y)&=&0,&
\end{array}
\right.
\end{equation}
the solution of which is intended to be defined in $[0,1]$.

\bigskip
\eqref{rescaledBVPcoupled} can be solved numerically through \emph{piecewise} orthogonal collocation \cite{elir00}. This is particularly useful when adaptive meshes might be needed to better follow the solution profile, and is now a standard approach (originally developed for ODEs \cite{amr88}, see \texttt{MatCont} \cite{matcont}). It means looking for $m$-degree piecewise continuous polynomials $\mu,\nu$ in $[0,1]$ and $w\in\mathbb{R}$ that satisfy the system
\begin{equation*}
\left\{\setlength\arraycolsep{0.1em}\begin{array}{rcll}
\mu(t_{i,j}) &=& F(\overline{\mu}_{t_{i,j}}\circ s_w,\overline{\nu}_{t_{i,j}}\circ s_w),&\quad j\in\{1,\ldots,m\},\quad i\in\{1,\ldots,L\}, \\[1mm]
\nu'(t_{i,j}) &=& \omega G(\overline{\mu}_{t_{i,j}}\circ s_w,\overline{\nu}_{t_{i,j}}\circ s_w),&\quad j\in\{1,\ldots,m\},\quad i\in\{1,\ldots,L\}, \\[1mm]
\mu(0) &=&\mu(1),&\\[1mm]
\nu(0) &=&\nu(1),&\\[1mm]
p(\mu,\nu)&=&0&
\end{array}
\right.
\end{equation*}
 having dimension $2d(1+Lm)+1$, for a given mesh $0= t_0<\cdots<t_L=1$ and collocation points $\{t_{i,j}\}_{i,j}$ such that $t_{i-1}  < t_{i,1}<\cdots<t_{i,m} < t_{i}$ for all $i\in\{1,\ldots,L\}$. The unknowns are, other than $w$, those of the form $\mu_{i,j}:=\mu(t_{i,j})$ and $\nu_{i,j}:=\nu(t_{i,j})$ for $(i,j)=(1,0)$ and $i\in\{1,\ldots,L\}$, $j\in\{1,\ldots,m\}$\footnote{In fact, one could also consider to represent $\mu$ and $\nu$ through their values at other sets of nodes as unknowns, not necessarily $\{t_{i,j}\}_{i,j}$ (see Remark \ref{r_representation} at the end of Section \ref{s_convergence}).}.
\begin{remark}
Typically, periodic solutions are computed within a continuation framework \cite{doe07}.  This provides a reasonable choice for the phase condition, necessary to ensure the actual well-posedness of \eqref{rescaledBVPcoupled}.
\end{remark}
\bigskip
As mentioned at the end of Subsection \ref{s_introre}, in applications from population dynamics right-hand sides usually feature an integral, therefore they cannot be exactly computed in general. This is also the case of \eqref{hpcoupled2}, which reads, once the time has been rescaled,
\begin{equation}\label{hprescaled}
\setlength\arraycolsep{0.1em}\begin{array}{rcl}
\tilde K(x_t\circ s_{\omega},y_t\circ s_{\omega})&=&\displaystyle\int_{-\frac{\tau}{\omega}}^0\omega K(\omega\theta,x_t(s_{\omega}(\omega\theta)),y_t(s_{\omega}(\omega\theta)))\dd\theta\\[4mm]
&=&\omega\displaystyle\int_{-\frac{\tau}{\omega}}^0 K(\omega\theta,x(t+\theta),y(t+\theta))\dd\theta,
\end{array}
\end{equation}
and the same holds for $\tilde H$, where now $x_t\in X:=B^{\infty}([-1,0],\mathbb{R}^d)$ and $y_t\in X:=B^{1,\infty}([-1,0],\mathbb{R}^d)$. Observe that, although the corresponding natural state space is in fact a Banach space of functions defined in $[-\tau/\omega,0]$, one could choose spaces of functions defined in $[-r,0]$ for any $1\geq r\geq\tau/\omega$\footnote{The extension of the state space to $[-1,0]$ is necessary since $\omega$ varies within an iterative procedure to find a numerical solution while the space needs to be fixed, as is required for the forthcoming analysis. Observe that such extension does not affect the dynamics: indeed, initial states that only differ in $[-1,-\tau/\omega]$ lead to the same orbits, but different orbits cannot cross.}.

\bigskip
Assuming that $K,H$ can be exactly computed, which is usually the case in applications, the approximation of \eqref{hprescaled} through quadrature as
\begin{equation}\label{quadrulek}
\tilde K_M(x_t\circ s_{\omega},y_t\circ s_{\omega}):=\omega \sum_{i=0}^Mw_i K(\omega\eta_i,x(t+\eta_i),y(t+\eta_i)), 
\end{equation}
and $\tilde H_M$ defined similarly, where $M$ is a given approximation level and $-\tau/\omega\leq\eta_0<\cdots<\eta_M\leq 0$, can also be exactly computed. Such an approximation corresponds to the \emph{secondary discretization} introduced in Subsection \ref{sub_numerical} and used in the convergence analysis that follows. The nodes $\eta_0,\ldots,\eta_M$ and the corresponding weights $w_0,\ldots,w_M$ are meant to define a suitable quadrature formula by exploiting possible irregularities in $K$, meaning that their choice does not need to be made a priori. Moreover, note that the quadrature nodes vary together with $\omega$, since the latter is unknown. In particular, they are completely independent of the collocation nodes mentioned earlier.
\section{Convergence analysis}
\label{s_convergence}
This section describes the theoretical convergence analysis of the numerical method described in Section \ref{s_piecewise}, following the abstract approach \cite{mas15NM}. In particular, the convergence analysis that follows applies to the Finite Element Method (FEM), which consists in letting $L\to\infty$ while keeping $m$ fixed and is the classical approach considered in practical applications (e.g., in \texttt{MatCont} \cite{matcont} or some versions of \texttt{DDE-Biftool} \cite{ddebiftool,sie15}).

\bigskip
\noindent Following the approach for RFDEs in \cite{ab20mas}, we first reformulate \eqref{rescaledBVPcoupled} as
\begin{equation}\label{convBVPcoupled}
\left\{\setlength\arraycolsep{0.1em}\begin{array}{rcll}
x(t) &=& F(x_t\circ s_{\omega},y_t\circ s_{\omega}),&\quad t\in[0,1], \\[1mm]
y'(t) &=& \omega G(x_t\circ s_{\omega},y_t\circ s_{\omega}),&\quad t\in[0,1], \\[1mm]
x_0 &=& x_1&\\[1mm]
y_0 &=& y_1&\\[1mm]
p(x\vert_{[0,1]})&=&0,&
\end{array}
\right.
\end{equation}
i.e., by imposing the periodicity condition to the states at the extrema of the period rather than to the solution values. In this case the solution $(x,y)$ is intended as a map defined in $[-1,1]$ and there is no need to resort to \eqref{perstate}.

\bigskip
Although formulations \eqref{rescaledBVPcoupled} and \eqref{convBVPcoupled} are formally different, they are mathematically equivalent and also lead to fundamentally equivalent numerical methods. Indeed, when applying the numerical method described in Section \ref{s_piecewise} to the problem \eqref{convBVPcoupled} one just introduces redundant variables.\footnote{A convergence analysis based on \eqref{rescaledBVPcoupled} can still be obtained for REs only by following the structure of the one presented here, see \cite{adena21}. For RFDEs, instead, this does not seem possible \cite{ab20mas}.}

\bigskip
The second step consists in observing that \eqref{convBVPcoupled} fits into the general form  of the BVP addressed in \cite{mas15NM}:
\begin{equation}\label{generalBVP}
\begin{cases}
u=\mathcal{F}(\mathcal{G}(u,\alpha),u,\beta),\\
\mathcal{B}(\mathcal{G}(u,\alpha),u,\beta)=0.
\end{cases}
\end{equation}
Here the relevant solution $v:=\mathcal{G}(u,\alpha)$ is assumed to lie in a normed space $\mathbb{V}$ of functions $[-1,1]\to\mathbb{R}^{2d}$, $u$ in a Banach space $\mathbb{U}$ of functions $[0,1]\to\mathbb{R}^{2d}$, and the operator $\mathcal{G}:\mathbb{U}\times\mathbb{A}\rightarrow\mathbb{V}$ represents a linear operator which reconstructs the solution given $u$ and an initial state $\alpha$, also lying in a Banach space $\mathbb{A}$ of functions $[-1,0]\to\mathbb{R}^{2d}$. $\beta$ is a vector of parameters living in a Banach space $\mathbb{B}$. 
The first line of \eqref{generalBVP} represents the concerned functional equation via the function $\mathcal{F}:\mathbb{V}\times\mathbb{U}\times\mathbb{B}\rightarrow\mathbb{U}$ and the second represents the boundary conditions via $\mathcal{B}:\mathbb{V}\times\mathbb{U}\times\mathbb{B}\rightarrow\mathbb{A}\times\mathbb{B}$.

\bigskip
In \cite{mas15NM}, \eqref{generalBVP} is then translated into a fixed-point problem, the so-called {\it Problem in Abstract Form} (PAF) which consists in finding $(v^{\ast},\beta^{\ast})\in \mathbb{V}\times\mathbb{B}$ with $v^{\ast}:=\mathcal{G}(u^{\ast},\alpha^{\ast})$ and $(u^{\ast},\alpha^{\ast},\beta^{\ast})\in \mathbb{U}\times\mathbb{A}\times\mathbb{B}$ such that 
\begin{equation}\label{PAF}
(u^{\ast},\alpha^{\ast},\beta^{\ast})=\Phi(u^{\ast},\alpha^{\ast},\beta^{\ast})
\end{equation}
for $\Phi:\mathbb{U}\times\mathbb{A}\times\mathbb{B}\rightarrow\mathbb{U}\times\mathbb{A}\times\mathbb{B}$ given by
\begin{equation}\label{Phi}
\Phi(u,\alpha,\beta):=
\begin{pmatrix}
\mathcal{F}(\mathcal{G}(u,\alpha),u,\beta)\\[2mm]
(\alpha,\beta)-\mathcal{B}(\mathcal{G}(u,\alpha),u,\beta)
\end{pmatrix}.
\end{equation}
In the sequel we always use the superscript $^{\ast}$ to denote quantities relevant to fixed points.

\bigskip
It follows that \eqref{convBVPcoupled} leads to an instance of \eqref{Phi} with $\mathcal{G}$, $\mathcal{F}$ and $\mathcal{B}$ given respectively by
\begin{equation}\label{G2}
\mathcal{G}_{}(u,\alpha)(t):=\begin{cases}
\displaystyle \begin{pmatrix} u_X(t)\\ \alpha_Y(0)+\int_0^t u_Y(t) \end{pmatrix},&t\in(0,1],\\[2mm]
\alpha(t),&t\in[-1,0],
\end{cases}
\end{equation}
\begin{equation}\label{Fco}
\mathcal{F}(v,u,\omega):= 
\begin{pmatrix}
F(v_{(\cdot)}\circ s_{\omega})\\[2mm]
\omega G(v_{(\cdot)}\circ s_{\omega})
\end{pmatrix}.
\end{equation}
and
\begin{equation}\label{Bco}
\mathcal{B}(v,u,\omega):=
\begin{pmatrix}
v_{0}-v_{1}\\[2mm]
p(v|_{[0,1]})
\end{pmatrix},
\end{equation}
where the subscripts $X,Y$ represent the first $d$ and the last $d$ components, respectively. The boundary operator is linear and independent of $\omega$.

\bigskip
The fact that \eqref{convBVPcoupled} can be rewritten as a PAF does not imply that the convergence framework in \cite{mas15NM} can be applied either way. In fact, several assumptions are required. These include theoretical assumptions, the validity of which depends on the choices of the spaces, as well as on the regularity of the integrands $K,H$ in \eqref{hpcoupled2}. Subsection \ref{sub_theoretical} includes the definitions of such assumptions and their statements as propositions, instanced according to the problems of interest. The other assumptions required concern instead the reduction of the problem to a finite-dimensional one, and will be dealt with, similarly, in Subsection \ref{sub_numerical}. Concerning the proofs of such propositions,  we will go through the main points, focusing on the differences with respect to the analogous propositions in the separated cases of RFDEs \cite{ab20mas} and REs \cite{ab22RE}, with the exception of a more complicated one to which we dedicate the entire \ref{s_appendix}.
\subsection{Theoretical assumptions}
\label{sub_theoretical}
With respect to \eqref{eq:coupled}, \eqref{hpcoupled1}, \eqref{hpcoupled2} and \eqref{convBVPcoupled}, the hypotheses needed to prove the validity of the theoretical assumptions in \cite{mas15NM} are:
\begin{enumerate}[label=(T\arabic*),ref=(T\arabic*)]
\item\label{T1} $\mathtt{X}=B^{\infty}([-\tau,0],\mathbb{R}^{d})$, $\mathtt{Y}=B^{1,\infty}([-\tau,0],\mathbb{R}^{d})$,\\ $X=B^{\infty}([-1,0],\mathbb{R}^{d})$, $Y=B^{1,\infty}([-1,0],\mathbb{R}^{d})$;
\item\label{T2} $\mathbb{U}=B^{\infty}([0,1],\mathbb{R}^{2d})$, $\mathbb{V}=B^{\infty}([-1,1],\mathbb{R}^{d})\times B^{1,\infty}([-1,1],\mathbb{R}^{d})$, \\$\mathbb{A}=B^{\infty}([-1,0],\mathbb{R}^{d})\times B^{1,\infty}([-1,0],\mathbb{R}^{d})$;
\item\label{T3} $K,H$ are piecewise continuous with partial derivatives $D_1 K$, $D_2 K$, $D_1H$, $D_2H$ which are measurable with respect to both their arguments;
\item\label{T4} $\tilde F,\tilde G$ are continuous with partial derivatives $D_1\tilde F$, $D_2\tilde F$, $D_1\tilde G$, $D_2\tilde G$;
\item\label{T5} $D_1\tilde F,D_1\tilde G\in\mathcal{C}(\R^d,\R^d)$, $D_2\tilde F,D_2\tilde G\in\mathcal{C}(\mathtt{Y},\R^d)$ and the maps $x\mapsto D_2K(q,x)$ and $x\mapsto D_2H(q,x)$ are piecewise continuous for all $q\in\R$;
\item\label{T6} there exist $r>0$ and $\kappa\geq0$ such that
\begin{equation*}
\left\{\setlength\arraycolsep{0.1em}\begin{array}{rcl}
\|D\tilde F(\tilde K(v_t\circ s_{\omega}), w_t\circ s_{\omega})&-&D\tilde F(\tilde K(v^*_t\circ s_{\omega^*}), w^*_t\circ s_{\omega^*})\|_{\R^d\leftarrow \R^d\times\mathtt{Y}}\\
&\leq &\kappa\|(v_t,\omega)-(v^{\ast}_{t}, \omega^{\ast})\|_{X\times Y\times\R}\\
\|D\tilde G(\tilde H(v_t\circ s_{\omega}), w_t\circ s_{\omega})&-&D\tilde G(\tilde H(v^*_t\circ s_{\omega^*}), w^*_t\circ s_{\omega^*})\|_{\R^d\leftarrow \R^d\times\mathtt{Y}}\\
&\leq& \kappa\|(v_t,\omega)-(v^{\ast}_{t}, \omega^{\ast})\|_{X\times Y\times\R},
\end{array}\right.
\end{equation*}
for every $(v_t,\omega)\in b((v^{\ast}_{t},\omega^{\ast}),r)$\footnote{$b(c,r)$ denotes the closed ball centered in $c$ having radius $r$.}, uniformly with respect to $t\in[0,1]$, where $w$ indicates the components of $v$ obtained from the RFDE.
\end{enumerate}

\bigskip
The first theoretical assumption (A$\mathfrak{F}\mathfrak{B}$, \cite[page 534]{mas15NM}) concerns the Fr\'echet-differen\-tiability of the operators $\mathcal{F}$ and $\mathcal{B}$ appearing in \eqref{Phi}. Since $p$ is linear, so is $\mathcal{B}$ in \eqref{Bco}, hence the latter is Fr\'echet-differentiable. The validity of the assumption is thus a direct consequence of the following.
\begin{proposition}\label{p_Afbco}
Under \ref{T1}-\ref{T4}, $\mathcal{F}$ in \eqref{Fco} is Fr\'echet-differentiable, from the right with respect to $\omega$, at every $(\hat v,\hat u,\hat\omega)\in\mathbb{V}\times\mathbb{U}\times(0,+\infty)$, and the first component of $D\mathcal{F}(\hat v,\hat u,\hat\omega)(v,u,\omega)$ is given by
\begin{equation}\label{DFco}
D\mathcal{F}_X(\hat v,\hat u,\hat\omega)(v,u,\omega):=\mathfrak{L}_X(\cdot;\hat v,\hat\omega)[v_{(\cdot)}\circ s_{\hat\omega}]+\omega\mathfrak{M}_X(\cdot;\hat v,\hat\omega)
\end{equation}
for $(v,u,\omega)\in\mathbb{V}\times\mathbb{U}\times(0,+\infty)$, where, for $t\in[0,1]$,
\begin{equation}\label{Lco1}
\setlength\arraycolsep{0.1em}\begin{array}{rl}
\mathfrak{L}_X(t;\hat v,\hat\omega)[v_{t}\circ s_{\hat\omega}]:=& D_1\tilde F(\tilde K(\hat v_{t}\circ s_{\hat\omega}),\hat w_t\circ s_{\hat\omega})L_X(t;\hat v,\hat\omega)[v_{t}\circ s_{\hat\omega}]\\[2mm]
+&\displaystyle D_2\tilde F(\tilde K(\hat v_{t}\circ s_{\hat\omega}),\hat w_t\circ s_{\hat\omega})[ w_t\circ s_{\hat\omega}]
\end{array}
\end{equation}
and
\begin{equation}\label{Mco1}
\setlength\arraycolsep{0.1em}\begin{array}{rl}
\mathfrak{M}_X(t;v,\omega):=&D_1\tilde F(\tilde K(v_t\circ s_{\omega}),w_t\circ s_{\omega})M_X(t;v,\omega)\\[2mm]
-&D_2\tilde F(\tilde K(v_{t}\circ s_{\omega}), w_t\circ s_{\omega})[w'_{t}\circ s_{\omega}]\cdot\frac{s_{\omega}}{\omega},
\end{array}
\end{equation}
and, in turn, 
\begin{equation*}
L_X(t;\hat v,\hat\omega)[v_t\circ s_{\hat\omega}]:=\hat\omega\int_{-\frac{\tau}{\hat\omega}}^0D_2K(\hat\omega\theta, \hat v(t+\theta),\hat w(t+\theta))[v(t+\theta),w(t+\theta)]\dd\theta
\end{equation*}
and 
\begin{equation*}
\setlength\arraycolsep{0.1em}\begin{array}{rl}
M_X(t;v,\omega):=&\displaystyle\int_{-\frac{\tau}{\omega}}^0 K(\omega\theta, v(t+\theta), w(t+\theta))\dd\theta\\[2mm]
-&\displaystyle\frac{\tau}{\omega} K\left(-\tau,v\left(t-\frac{\tau}{\omega}\right),w\left(t-\frac{\tau}{\omega}\right)\right)\\[2mm]
+&\displaystyle\omega\int_{-\frac{\tau}{\omega}}^0D_1K(\omega\theta,v(t+\theta),w(t+\theta))\theta\dd\theta.
\end{array}
\end{equation*}
The second component $D\mathcal{F}_Y(\hat v,\hat u,\hat\omega)(v,u,\omega)$ is defined similarly from
\begin{equation}\label{Lco2}
\setlength\arraycolsep{0.1em}\begin{array}{rl}
\mathfrak{L}_Y(t;\hat v,\hat\omega)[v_{t}\circ s_{\hat\omega}]:=\hat\omega[& D_1\tilde G(\tilde H(\hat v_{t}\circ s_{\hat\omega}),\hat w_t\circ s_{\hat\omega})L_Y(t;\hat v,\hat\omega)[v_{t}\circ s_{\hat\omega}]\\[2mm]
+&\displaystyle D_2\tilde G(\tilde H(\hat v_{t}\circ s_{\hat\omega}),\hat w_t\circ s_{\hat\omega})[w_t\circ s_{\hat\omega}]]
\end{array}
\end{equation}
and
\begin{equation}\label{Mco2}
\setlength\arraycolsep{0.1em}\begin{array}{rl}
\mathfrak{M}_Y(t;v,\omega):=G(v_{t}\circ s_{\omega})+\omega\bigg[&D_1\tilde G(\tilde H(v_t\circ s_{\omega}),w_t\circ s_{\omega})M_Y(t;v,\omega)\\[2mm]
-&D_2\tilde G(\tilde H(v_{t}\circ s_{\omega}), w_t\circ s_{\omega})[w'_{t}\circ s_{\omega}]\cdot\frac{s_{\omega}}{\omega}\bigg],
\end{array}
\end{equation}
where $L_Y$ and $M_Y$ are defined similarly to $L_X$ and $M_X$.

\end{proposition}

\begin{proof} 
 Basically, the expression \eqref{DFco}, defined through \eqref{Lco1} and \eqref{Mco1}, is directly proven to satisfy the definition of differentiable function according to \cite[Definition 1.1.1]{ampr95}, and the same is done for the second component of $D\mathcal{F}$. It is worth pointing out that assuming hypotheses such as \eqref{hpcoupled1} and \eqref{hpcoupled2}, which is anyway typical in applications from population dynamics, is crucial in this proposition for the RE component. Basically, for the thesis to hold it is required that the right-hand side always lies in a more regular space than $\mathbb{U}$ (which is always the case for RFDEs, where $\mathbb{U}$ plays the role of the space of the derivatives). This can be observed by looking at the last addend of \cite[(2.12)]{ab20mas}, where the derivative of the state of an element of $\mathbb{V}$ appears as a factor. The same factor would be a problem without any assumption whatsoever on $F$ and $G$ since the first component of $\mathbb{V}$ is as regular as that of $\mathbb{U}$ under \ref{T2} in the case of coupled problems.\end{proof}

\bigskip
The second theoretical assumption (A$\mathfrak{G}$, \cite[page 534]{mas15NM}) concerns the boundedness of $\mathcal{G}$ defined in \eqref{G2}. The following proposition concerns its validity, and its proof is an immediate consequence of the definition \eqref{G2}.

\begin{proposition}\label{l_G}
Under (T2), $\mathcal{G}_{}$ is bounded.
\end{proposition}

\bigskip
\noindent The third theoretical assumption (A$x^*1$, \cite[page 536]{mas15NM}) concerns the local Lipschitz continuity of the Fr\'echet derivative of the fixed point operator $\Phi$ in \eqref{Phi} at the relevant fixed points. In the sequel $(u^{\ast},\alpha^{\ast},\omega^{\ast})\in\mathbb{U}_{}\times\mathbb{A}_{}\times(0,+\infty)$ is a fixed point of $\Phi_{}$ and $(x^{\ast},y^{\ast})$ is the corresponding $1$-periodic solution of \eqref{eq:coupled}. With respect to the validity of Assumption A$x^*1$ the following holds.
\begin{proposition}\label{p_Ax*1co}
Under \ref{T1}-\ref{T4} and \ref{T6}, there exist $r\in(0,\omega^{\ast})$ and $\kappa\geq0$ such that
\begin{equation*}
\setlength\arraycolsep{0.1em}\begin{array}{rcl}
\|D\Phi(u,\alpha,\omega)&-&D\Phi(u^{\ast},\alpha^{\ast},\omega^{\ast})\|_{\mathbb{U}\times\mathbb{A}\times\R\leftarrow\mathbb{U}\times\mathbb{A}\times(0,+\infty)}\\[2mm]
&\leq&\kappa\|(u,\alpha,\omega)-(u^{\ast},\alpha^{\ast},\omega^{\ast})\|_{\mathbb{U}\times\mathbb{A}\times\R}
\end{array}
\end{equation*}
for all $(u,\alpha,\omega)\in b((u^{\ast},\alpha.^{\ast},\omega^{\ast}),r)$.
\end{proposition} 

\begin{proof} 
Just as its proof for the RFDE case in \cite{ab20mas} and that for the RE case in \cite{ab22RE}, the proposition can be proved thanks to the fact that  $u^*$ lies in fact in a more regular subspace of its space $\mathbb{U}$, which is a consequence of the assumptions \eqref{hpcoupled1}-\eqref{hpcoupled2}.\end{proof}

The fourth (and last) theoretical assumption (A$x^{\ast}2$, \cite[page 536]{mas15NM}), concerns the well-posedness of a linearized inhomogeneous version of the PAF \eqref{PAF}. Its validity can be proved under (T1), (T2), (T3) and (T4), together with an additional requirement, which in turn follows from assuming, e.g., the {\it hyperbolicity} of the periodic solution\footnote{Let us remark that the condition of hyperbolicity is necessary for the local stability analysis of periodic solutions in view of the Principle of Linearized Stability \cite{bl20}.} of the original problem.
It is convenient to introduce the abbreviations
\begin{equation*}
\mathfrak{L}_{}^{\ast}:=\mathfrak{L}_{}(\cdot;v^{\ast},\omega^{\ast}),\qquad\mathfrak{M}_{}^{\ast}:=\mathfrak{M}_{}(\cdot;v^{\ast},\omega^{\ast}).
\end{equation*}
\begin{proposition}\label{p_Ax*2co}
Under \ref{T1}-\ref{T5}, let $T^*(t,s):X\times Y\to X\times Y$ be the evolution operator for the linear homogeneous delay equation
\begin{equation*}
(x(t),y'(t))=\mathfrak{L}(t;v^{\ast},\omega^{\ast})[(x_{t},y_t)\circ s_{\omega^{\ast}}].
\end{equation*}
If the eigenvalue $1\in\sigma(T^{\ast}(1,0))$ is simple, then the linear bounded operator $I_{\mathbb{U}\times\mathbb{A}\times\mathbb{B}}-D\Phi(u^{\ast},\alpha^{\ast},\omega^{\ast})$ is invertible, i.e., for all $(u_{0},\alpha_{0},\omega_{0})\in\mathbb{U}\times\mathbb{A}\times\mathbb{B}$ there exists a unique $(u,\alpha,\omega)\in\mathbb{U}\times\mathbb{A}\times\mathbb{B}$ such that
\begin{equation*}
\left\{\setlength\arraycolsep{0.1em}\begin{array}{l}
u=\mathfrak{L}^{\ast}[\mathcal{G}(u,\alpha)_{\cdot}\circ s_{\omega^{\ast}}]+\omega\mathfrak{M}^{\ast}+u_{0}\\[2mm]
\alpha=\mathcal{G}(u,\alpha)_{1}+\alpha_{0}\\[2mm]
p(\mathcal{G}(u,\alpha)|_{[0,1]})=\omega_{0}.
\end{array}
\right.
\end{equation*}
\end{proposition} 
\begin{proof}
Using the variation of constant formula for RFDEs and arguing as in \cite[Proposition 3.4]{ab22RE} for the RE component, from the periodicity condition one obtains an equality of the form
\begin{equation*}
\alpha=T^{\ast}(1,0)\alpha+\omega\xi_1^*+\xi_2^*,
\end{equation*}
where the term $\xi_1^*$ depends only on the linearized problem (i.e., on $\mathfrak{M}^*$) while $\xi_2^*$ depends on $u_0$ and $\varphi_0$. By the hypothesis on the multiplier 1, one can decompose the space $X\times Y=R\oplus K$ as in \cite[Proposition 3.4]{ab22RE}. By assuming that $\xi_1^*\not\in R$ (the proof of which is given in \ref{s_appendix}, as anticipated in the introduction), one gets a unique solution $\omega$, and therefore a unique possible value of $\alpha-T^{\ast}(1,0)\alpha$.
\end{proof}
\subsection{Numerical assumptions}
\label{sub_numerical}
As anticipated, the present subsection deals with the numerical assumptions, which concern the chosen discretization scheme for the numerical method. Such scheme is defined by the {\it primary} and the {\it secondary} discretizations.

\bigskip
As in \cite{ab20mas}, the primary discretization consists in reducing the spaces $\mathbb{U}$ and $\mathbb{A}$ to finite-dimensional spaces $\mathbb{U}_{L}$ and $\mathbb{A}_{L}$, given a level of discretization $L$. This happens by means of {\it restriction} operators $\rho_{L}^{+}:\mathbb{U}\rightarrow\mathbb{U}_{L}$, $\rho_{L}^{-}:\mathbb{A}\rightarrow\mathbb{A}_{L}$ and {\it prolongation} operators $\pi_{L}^{+}:\mathbb{U}_{L}\rightarrow\mathbb{U}$, $\pi_{L}^{-}:\mathbb{A}_{L}\rightarrow\mathbb{A}$, which extend respectively to $R_{L}:\mathbb{U}\times\mathbb{A}\times\mathbb{B}\rightarrow\mathbb{U}_{L}\times\mathbb{A}_{L}\times\mathbb{B}$ given by $R_{L}(u,\alpha,\omega):=(\rho_{L}^{+}u,\rho_{L}^{-}\alpha,\omega)$ and $P_{L}:\mathbb{U}_{L}\times\mathbb{A}_{L}\times\mathbb{B}\rightarrow\mathbb{U}\times\mathbb{A}\times\mathbb{B}$ given by $P_{L}(u_{L},\alpha_{L},\omega):=(\pi_{L}^{+}u_{L},\pi_{L}^{-}\alpha_{L},\omega)$.
All of them are linear and bounded. In the following we describe the specific choices we make in this context, based on piecewise polynomial interpolation.

Starting from $\mathbb{U}$, which concerns the interval $[0,1]$, we choose the uniform {\it outer} mesh
\begin{equation}\label{outmesh+}
\Omega_{L}^{+}:=\{t_{i}^{+}=ih\ :\ i\in\{0,\ldots,L\},\,h=1/L\}\subset[0,1],
\end{equation}
and {\it inner} meshes
\begin{equation}\label{inmesh+}
\Omega_{L,i}^{+}:=\{t_{i,j}^{+}:=t_{i-1}^{+}+c_{j}h\ :\ j\in\{1,\ldots,m\}\}\subset[t_{i-1}^{+},t_{i}^{+}],\quad i\in\{1,\ldots,L\},
\end{equation}
where $0<c_{1}<\cdots<c_{m}<1$ are given abscissae for $m$ a positive integer. Correspondingly, we define
\begin{equation}\label{UL}
\mathbb{U}_{L}:=\mathbb{R}^{2d(1+Lm)},
\end{equation}
whose elements $u_{L}$ are indexed as
\begin{equation}\label{uL}
u_{L}:=(u_{1,0},u_{1,1},\ldots,u_{1,m},\ldots,u_{L,1}\ldots,u_{L,m})^{T}
\end{equation}
with components in $\mathbb{R}^{2d}$. Finally, we define, for $u\in\mathbb{U}$,
\begin{equation}\label{rL+}
\rho_{L}^{+}u:=(u(0),u(t_{1,1}^{+}),\ldots,u(t_{1,m}^{+}),\ldots,u(t_{L,1}^{+})\ldots,u(t_{L,m}^{+}))^{T}\in\mathbb{U}_{L}
\end{equation}
and, for $u_{L}\in\mathbb{U}_{L}$, $\pi_{L}^{+}u_{L}\in\mathbb{U}$ as the unique element of the space
\begin{equation}\label{PiLm+}
\Pi_{L,m}^{+}:=\{p\in C([0,1],\mathbb{R}^{2d})\ :\ p\vert_{[t_{i-1}^{+},t_{i}^{+}]}\in\Pi_m,\;i\in\{1,\ldots,L\}\}
\end{equation}
such that
\begin{equation}\label{pL+}
\pi_{L}^{+}u_{L}(0)=u_{1,0},\quad\pi_{L}^{+}u_{L}(t_{i,j}^{+})=u_{i,j},\quad j\in\{1,\ldots,m\},\;i\in\{1,\ldots,L\}.
\end{equation}
Above $\Pi_{m}$ is the space of $\mathbb{R}^{2d}$-valued polynomials having degree $m$ and, when needed, we represent $p\in\Pi_{L,m}^{+}$ through its pieces as
\begin{equation}\label{lagrange}
p\vert_{[t_{i-1}^{+},t_{i}^{+}]}(t)=\sum_{j=0}^{m}\ell_{m,i,j}(t)p(t_{i,j}^{+}),\quad t\in[0,1],
\end{equation}
where, for ease of notation, we implicitly set
\begin{equation}\label{t0+}
t_{i,0}^{+}:=t_{i-1}^{+},\quad i\in\{1,\ldots,L\},
\end{equation}
and $\{\ell_{m,i,0},\ell_{m,i,1},\ldots,\ell_{m,i,m}\}$ is the Lagrange basis relevant to the nodes $\{t_{i,0}^{+}\}\cup\Omega_{L,i}^{+}$. Observe that the latter is invariant with respect to $i$ as long as we fix the abscissae $c_{j}$, $j\in\{1,\ldots,m\}$, defining the inner meshes \eqref{inmesh+}. Indeed, for every $i\in\{1,\ldots,L\}$,
\begin{equation*}
\ell_{m,i,j}(t)=\ell_{m,j}\left(\frac{t-t_{i-1}^{+}}{h}\right),\quad t\in[t_{i-1}^{+},t_{i}^{+}],
\end{equation*}
where $\{\ell_{m,0},\ell_{m,1},\ldots,\ell_{m,m}\}$ is the Lagrange basis in $[0,1]$ relevant to the abscissae $c_{0},c_{1},\ldots,c_{m}$ with $c_{0}:=0$. 

Similarly, for $\mathbb{A}$, which concerns the interval $[-1,0]$, we choose
\begin{equation}\label{outmesh-}
\Omega_{L}^{-}:=\{t_{i}^{-}=ih-1\ :\ i\in\{0,\ldots,L\},\;h=1/L\}\subset[-1,0],
\end{equation}
and
\begin{equation}\label{inmesh-}
\Omega_{L,i}^{-}:=\{t_{i,j}^{-}:=t_{i-1}^{-}+c_{j}h\ :\ j\in\{1,\ldots,m\}\}\subset[t_{i-1}^{-},t_{i}^{-}],\quad i\in\{1,\ldots,L\}.
\end{equation}
Correspondingly, we define
\begin{equation}\label{AL}
\mathbb{A}_{L}:=\mathbb{R}^{2(1+Lm)\times d}
\end{equation}
with indexing
\begin{equation}\label{psiL}
\alpha_{L}:=(\alpha_{1,0},\alpha_{1,1},\ldots,\alpha_{1,m},\ldots,\alpha_{L,1}\ldots,\alpha_{L,m})^{T};
\end{equation}
for $\alpha\in\mathbb{A}$,
\begin{equation}\label{rL-}
\rho_{L}^{-}\alpha:=(\alpha(-1),\alpha(t_{1,1}^{-}),\ldots,\alpha(t_{1,m}^{-}),\ldots,\alpha(t_{L,1}^{-})\ldots,\alpha(t_{L,m}^{-}))^{T}\in\mathbb{A}_{L}
\end{equation}
and, for $\alpha_{L}\in\mathbb{A}_{L}$, $\pi_{L}^{-}\alpha_{L}\in\mathbb{A}$ as the unique element of the space
\begin{equation}\label{PiLm-}
\Pi_{L,m}^{-}:=\{p\in C([-1,0],\mathbb{R}^{2d})\ :\ p\vert_{[t_{i-1}^{-},t_{i}^{-}]}\in\Pi_m,\;i\in\{1,\ldots,L\}\}
\end{equation}
such that
\begin{equation}\label{pL-}
\pi_{L}^{-}\alpha_{L}(-1)=\alpha_{1,0},\quad\pi_{L}^{-}\alpha_{L}(t_{i,j}^{-})=\alpha_{i,j},\quad j\in\{1,\ldots,m\},\;i\in\{1,\ldots,L\}.
\end{equation}
Elements in $\Pi_{L,m}^{-}$ are represented in the same way as those of $\Pi_{L,m}^{+}$ by suitably adapting both \eqref{lagrange} and \eqref{t0+}.
\begin{remark}\label{r_primary}
It is worth pointing out that more general choices can be made concerning outer and inner meshes. In particular, as already remarked in Section \ref{s_piecewise}, in practical applications {\it adaptive} outer meshes represent a standard for RFDEs, see, e.g., \cite{elir00}. As for inner meshes, abscissae including the extrema of $[0,1]$ can also be considered, paying attention to put the correct constraints at the internal outer nodes, i.e., $t_{i}^{\pm}$ for $i\in\{1,\ldots,L-1\}$.
\end{remark}
\bigskip
The secondary discretization consists in replacing $\mathcal{F}$ in the first of \eqref{Phi} with an operator $\mathcal{F}_{M}$ that can be exactly computed, for a given level of discretization $M$. In particular, we define $\mathcal{F}_{M}$ through $\tilde{K}_{M},\tilde{H}_M$ defined through \eqref{quadrulek}. Correspondingly, $F_M$ and $G_M$ are defined from suitable discrete versions $\tilde{F}_M$ and $\tilde{G}_M$ of $\tilde{F}$ and $\tilde{G}$, respectively. A secondary discretization for $\mathcal{G}$ in \eqref{Phi} is instead unnecessary, since it can be evaluated exactly in $\pi_{L}^{+}\mathbb{U}_{L}\times\pi_{L}^{-}\mathbb{A}_{L}$ according to \eqref{UL} and \eqref{AL}. As for the operator $p$ defining the phase condition in \eqref{Phi}, we assume that it can be evaluated exactly in $\pi_{L}^{+}\mathbb{U}_{L}$\footnote{This is indeed true in the case of integral phase conditions if the piecewise quadrature is based on the mesh of the primary discretization, which is the standard approach in practical applications.}. Eventually, $\Phi_{M}$ is the operator obtained by replacing $\mathcal{F}$ in $\Phi$ with its approximated version, i.e., $\Phi_{M}:\mathbb{U}\times\mathbb{A}\times\mathbb{B}\rightarrow\mathbb{U}\times\mathbb{A}\times\mathbb{B}$ defined by
\begin{equation}\label{PhiMco}
\Phi_{M}(u,\alpha,\omega):=
\begin{pmatrix}
F_{M}(\mathcal{G}(u,\alpha)_{(\cdot)}\circ s_{\omega})\\[2mm]
\omega G_{M}(\mathcal{G}(u,\alpha)_{(\cdot)}\circ s_{\omega})\\[2mm]
\mathcal{G}(u,\alpha)_{1}\\[2mm]
\omega-p(\mathcal{G}(u,\alpha)|_{[0,1]})
\end{pmatrix}.
\end{equation}

\bigskip
From the two discretizations together we can define the discrete version 
\begin{equation*}
\Phi_{L,M}:=R_{L}\Phi_{M} P_{L}:\mathbb{U}_{L}\times\mathbb{A}_{L}\times\mathbb{B}\rightarrow\mathbb{U}_{L}\times\mathbb{A}_{L}\times\mathbb{B}
\end{equation*}
of the fixed point operator $\Phi$ in \eqref{Phi} as
\begin{equation}\label{eq:philm}
\Phi_{L,M}(u_{L},\alpha_{L},\omega):=
\begin{pmatrix}
\rho_{L}^{+}F_{M}(\mathcal{G}(\pi_{L}^{+}u_{L},\pi_{L}^{-}\alpha_{L})_{(\cdot)}\circ s_{\omega})\\[2mm]
\omega\rho_{L}^{+}G_{M}(\mathcal{G}(\pi_{L}^{+}u_{L},\pi_{L}^{-}\alpha_{L})_{(\cdot)}\circ s_{\omega})\\[2mm]
\rho_{L}^{-}(\mathcal{G}(\pi_{L}^{+}u_{L},\pi_{L}^{-}\alpha_{L})_1\\[2mm]
\omega-p(\mathcal{G}(\pi_{L}^{+}u_{L},\pi_{L}^{-}\alpha_{L})
\end{pmatrix}.
\end{equation}
A fixed point $(u_{L,M}^{\ast},\alpha_{L,M}^{\ast},\omega_{L,M}^{\ast})$ of $\Phi_{L,M}$ can be found by standard solvers for nonlinear systems of algebraic equations\footnote{{In particular, we use Newton's method, see Section \ref{s_results}.}} and, as will be shown in Subsection \ref{s_convresult}, its prolongation $P_{L}(u_{L,M}^{\ast},\alpha_{L,M}^{\ast},\omega_{L,M}^{\ast})$ is then considered as an approximation of a fixed point $(u^{\ast},\alpha^{\ast},\omega^{\ast})$ of $\Phi$ in \eqref{Phi}. Correspondingly, $v_{L,M}^{\ast}:=\mathcal{G}(\pi_{L}^{+}u_{L,M}^{\ast},\pi_{L}^{-}\alpha_{L,M}^{\ast})$ is considered as an approximation of the solution $v^{\ast}=\mathcal{G}(u^{\ast},\alpha^{\ast})$ of \eqref{Phi}.

\bigskip
The hypotheses on the discretization method needed to prove the validity of the numerical assumptions in \cite{mas15NM} are:
\begin{enumerate}[label=(N\arabic*),ref=(N\arabic*)]
\item\label{N1} the primary discretization of the space $\mathbb{U}$ is based on the choices \eqref{outmesh+}--\eqref{pL+};
\item\label{N2} the primary discretization of the space $\mathbb{A}$ is based on the choices \eqref{outmesh-}--\eqref{pL-};
\item\label{N3} the nodes $\eta_0,\ldots,\eta_M$, together with the weights $w_0,\ldots,w_M$ chosen for the secondary discretization as in \eqref{quadrulek} define a (piecewise) interpolatory quadrature formula which is convergent in $B^{\infty}([0,1],\R^d)$;
\item\label{N4} $\tilde F_M,\tilde G_M:\R^d\times \mathtt{Y}\to\R^d$ are continuous with partial derivatives $D_1\tilde F_M$, $D_2\tilde F_M$, $D_1\tilde G_M$, $D_2\tilde G_M$;
\item\label{N5} $D_1\tilde F_M,D_1\tilde G_M\in\mathcal{C}(\R^d,\R^d)$ and $D_2\tilde F_M,D_2\tilde G_M\in\mathcal{C}(\mathtt{Y},\R^d)$;

\item\label{N6} there exist $r>0$ and $\kappa\geq0$ such that
\begin{equation*}
\left\{\setlength\arraycolsep{0.1em}\begin{array}{rcl}
\|D\tilde F_M(\tilde K_M(v_t\circ s_{\omega}), w_t\circ s_{\omega})&-&D\tilde F_M(\tilde K_M(v^*_t\circ s_{\omega^*}), w^*_t\circ s_{\omega^*})\|_{\R^d\leftarrow \R^d\times\mathtt{Y}}\\
&\leq &\kappa\|(v_t,\omega)-(v^{\ast}_{t}, \omega^{\ast})\|_{X\times Y\times\R}\\
\|D\tilde G_M(\tilde H_M(v_t\circ s_{\omega}), w_t\circ s_{\omega})&-&D\tilde G_M(\tilde H_M(v^*_t\circ s_{\omega^*}), w^*_t\circ s_{\omega^*})\|_{\R^d\leftarrow \R^d\times\mathtt{Y}}\\
&\leq& \kappa\|(v_t,\omega)-(v^{\ast}_{t}, \omega^{\ast})\|_{X\times Y\times\R},
\end{array}\right.
\end{equation*}
for every $(v_t,\omega)\in b((v^{\ast}_{t},\omega^{\ast}),r)$, uniformly with respect to $t\in[0,1]$, where $w$ indicates the components of $v$ obtained from the RFDE;
\item\label{N7} the following hold uniformly with respect to $t\in[0,1]$:
\begin{itemize}
\item $\displaystyle\lim_{M\rightarrow\infty}|F_{M}(\tilde K_M(v^*_t\circ s_{\omega^*}), w^*_t\circ s_{\omega^*})-F(\tilde K_M(v^*_t\circ s_{\omega^*}), w^*_t\circ s_{\omega^*})|=0$;
\item $\displaystyle\lim_{M\rightarrow\infty}|G_{M}(\tilde H_M(v^*_t\circ s_{\omega^*}), w^*_t\circ s_{\omega^*})-G(\tilde H_M(v^*_t\circ s_{\omega^*}), w^*_t\circ s_{\omega^*})|=0$;
\end{itemize}
\item\label{N8} the following hold uniformly with respect to $t\in[0,1]$:
\begin{itemize}
\item $\setlength\arraycolsep{0.1em}\begin{array}{rl}\displaystyle\lim_{M\rightarrow\infty}\|&D\tilde F_{M}(\tilde K_M(v^*_t\circ s_{\omega^*}), w^*_t\circ s_{\omega^*})\\
&-DF(\tilde K_M(v^*_t\circ s_{\omega^*}), w^*_t\circ s_{\omega^*}\|_{\mathbb{R}^d\leftarrow\mathbb{R}\times\mathtt{Y}}=0\end{array}$;
\item $\setlength\arraycolsep{0.1em}\begin{array}{rl}\displaystyle\lim_{M\rightarrow\infty}\|&D\tilde G_{M}(\tilde H_M(v^*_t\circ s_{\omega^*}), w^*_t\circ s_{\omega^*})\\
&-D\tilde G(\tilde H_M(v^*_t\circ s_{\omega^*}), w^*_t\circ s_{\omega^*})\|_{\mathbb{R}^d\leftarrow\mathbb{R}\times\mathtt{Y}}=0\end{array}$.
\end{itemize}
\end{enumerate}
{Let us remark that \ref{N4}-\ref{N8} are not restrictive for common models and standard secondary discretizations (as, e.g., \eqref{quadrulek}).}

\bigskip
The first numerical assumption to be verified in \cite{mas15NM} is Assumption A$\mathfrak{F}_K\mathfrak{B}_K$ (page 535). As already observed, $\mathcal{B}$ and $p$ are linear functions, thus the proof of its validity is a direct consequence of the following.
\begin{proposition}\label{p_AfMbco}
Under \ref{T1}-\ref{T3} and \ref{N4} there exists $r\in(0,\omega^{\ast})$ such that $\mathcal{F}_{M}$ is Fr\'echet-differentiable, from the right with respect to $\omega$, at every point $(\hat{v},\hat u,\hat\omega)\in b((v^{\ast},u^{\ast},\omega^{\ast}),r)$. The first component of the differential $D\mathcal{F}_M$ is given by
\begin{equation*}
D\mathcal{F}_{M,X}(\hat v,\hat u,\hat\omega)(v,u,\omega)=\mathfrak{L}_{M,X}(\cdot;\hat v,\hat\omega)[v_{(\cdot)}\circ s_{\hat\omega}]+\omega\mathfrak{M}_{M,X}(\cdot;\hat v,\hat\omega)
\end{equation*}
for $(v,u,\omega)\in\mathbb{V}\times\mathbb{U}\times(0,+\infty)$, where, for $t\in[0,1]$,
\begin{equation*}
\setlength\arraycolsep{0.1em}\begin{array}{rl}
\mathfrak{L}_{M,X}(t;\hat v,\hat\omega)[v_{t}\circ s_{\hat\omega}]:=& D_1\tilde F_M(\tilde K_M(\hat v_{t}\circ s_{\hat\omega}),\hat w_t\circ s_{\hat\omega})L_{M,X}(t;\hat v,\hat\omega)[v_{t}\circ s_{\hat\omega}]
\end{array}
\end{equation*}
and
\begin{equation*}
\setlength\arraycolsep{0.1em}\begin{array}{rl}
\mathfrak{M}_{M,X}(t;v,\omega):=&D_1\tilde F_M(\tilde K_M(v_t\circ s_{\omega}),w_t\circ s_{\omega})M_{M,X}(t;v,\omega)\\[2mm]
-&D_2\tilde F_M(\tilde K_M(v_{t}\circ s_{\omega}), w_t\circ s_{\omega})[v'_{t}\circ s_{\omega}]\cdot\frac{s_{\omega}}{\omega},
\end{array}
\end{equation*}
and, in turn, $L_{M,X}$ and $M_{M,X}$ are defined as $\mathfrak{L}_M$ and $\mathfrak{M}_M$ in the RE case, respectively.
The second component of $D\mathcal{F}_M(\hat v,\hat u,\hat\omega)(v,u,\omega)$ is defined similarly from
\begin{equation*}
\setlength\arraycolsep{0.1em}\begin{array}{rl}
\mathfrak{L}_{M,Y}(t;\hat v,\hat\omega)[v_{t}\circ s_{\hat\omega}]:=\hat\omega[& D_1\tilde G_M(\tilde H_M(\hat v_{t}\circ s_{\hat\omega}),\hat w_t\circ s_{\hat\omega})L_{M,Y}(t;\hat v,\hat\omega)[v_{t}\circ s_{\hat\omega}]\\[2mm]
+&\displaystyle D_2\tilde G_M(\tilde H_M(\hat v_{t}\circ s_{\hat\omega}),\hat w_t\circ s_{\hat\omega})[\hat v_t\circ s_{\hat\omega}]]
\end{array}
\end{equation*}
and
\begin{equation*}
\setlength\arraycolsep{0.1em}\begin{array}{rcl}
\mathfrak{M}_{M,Y}(t;v,\omega)&:=&G_M(v_{t}\circ s_{\omega})\\&&+\omega\bigg[D_1\tilde G_M(\tilde H_M(v_t\circ s_{\omega}),w_t\circ s_{\omega})M_{M,Y}(t;v,\omega)\\[2mm]
&&-D_2\tilde G_M(\tilde H_M(v_{t}\circ s_{\omega}), w_t\circ s_{\omega})[v'_{t}\circ s_{\omega}]\cdot\frac{s_{\omega}}{\omega}\bigg],
\end{array}
\end{equation*}
where $L_{M,Y}$ and $M_{M,Y}$ are defined similarly to $L_{M,X}$ and $M_{M,X}$.
\end{proposition}
\begin{proof}The proposition can be proved as Proposition \ref{p_Afbco}, by replacing $\mathcal{F}$ in the first of \eqref{Phi} with $\mathcal{F}_M$.\end{proof}

\bigskip
For the remaining numerical assumptions, it is useful to define $\Psi_{L,M}:\mathbb{U}\times\mathbb{A}\times\mathbb{B}\rightarrow\mathbb{U}\times\mathbb{A}\times\mathbb{B}$ as
\begin{equation}\label{PsiLM}
\Psi_{L,M}:=I_{\mathbb{U}\times\mathbb{A}\times\mathbb{B}}-P_{L}R_{L}\Phi_{M}.
\end{equation}

\bigskip
The second numerical assumption in \cite{mas15NM} is CS1 (page 536), which is somehow the discrete version of A$x^{\ast}1$ therein, here Proposition \ref{p_Ax*1co}. With respect to its validity, the following holds.
\begin{proposition}\label{p_CS1co}
Under \ref{T1}-\ref{T4}, \ref{N1}-\ref{N4} and \ref{N6}, there exist $r_{1}\in(0,\omega^{\ast})$ and $\kappa\geq0$ such that
\begin{equation*}
\setlength\arraycolsep{0.1em}\begin{array}{rcl}
\|D\Psi_{L,M}(u,\alpha,\omega)&-&D\Psi_{L,M}(u^{\ast},\alpha^{\ast},\omega^{\ast})\|_{\mathbb{U}\times\mathbb{A}\times\mathbb{B}\leftarrow\mathbb{U}_{}\times\mathbb{A}\times(0,+\infty)}\\[2mm]
&\leq&\kappa\|(u,\alpha,\omega)-(u^{\ast},\psi^{\ast},\alpha^{\ast})\|_{\mathbb{U}\times\mathbb{A}\times\mathbb{B}}
\end{array}
\end{equation*}
for all $(u,\psi,\omega)\in b((u^{\ast},\psi^{\ast},\omega^{\ast}),r_{1})$ and for all positive integers $L$ and $M$.
\end{proposition}

\begin{proof} The proof is obtained by putting together that of \cite[Proposition 3.7]{ab20mas} and that of \cite[Proposition 3.7]{ab22RE}.\end{proof}

\bigskip
The last numerical assumption in \cite{mas15NM} (CS2, page 537), can be seen as the discrete version of A$x^{\ast}2$ therein, here Proposition \ref{p_Ax*2co}. With respect to its validity, the following holds.
\begin{proposition}\label{p_CS2b}
Under \ref{T1}-\ref{T5}, \ref{N1}-\ref{N5}, \ref{N7} and \ref{N8}, the operator $D\Psi_{L,M}(u^{\ast},\alpha^{\ast},\omega^{\ast})$ is invertible and its inverse is uniformly bounded with respect to both $L$ and $M$. Moreover,
\begin{equation*}
\setlength\arraycolsep{0.1em}\begin{array}{rcl}
\displaystyle\lim_{L,M\rightarrow\infty}&&\displaystyle\frac{1}{r_{2}(L,M)}\|[D\Psi_{L,M}(u^{\ast},\alpha^{\ast},\omega^{\ast})]^{-1}\|_{\mathbb{U}\times\mathbb{A}\times\mathbb{B}\leftarrow\mathbb{U}\times\mathbb{A}\times\mathbb{B}}\\[3mm]
&&\cdot\|\Psi_{L,M}(u^{\ast},\alpha^{\ast},\omega^{\ast})\|_{\mathbb{U}\times\mathbb{A}\times\mathbb{B}}=0,
\end{array}
\end{equation*}
where
\begin{equation*}
r_{2}(L,M):=\min\left\{r_{1},\frac{1}{2\kappa\|[D\Psi_{L,M}(u^{\ast},\alpha^{\ast},\omega^{\ast})]^{-1}\|_{\mathbb{U}\times\mathbb{A}\times\mathbb{B}\leftarrow\mathbb{U}\times\mathbb{A}\times\mathbb{B}}}\right\}
\end{equation*}
with $r_{1}$ and $\kappa$ as in Proposition \ref{p_CS1co}.
\end{proposition}
\begin{proof}The first step of the proof concerns the invertibility of the operator $D\Psi_{L,M}(u^*,\alpha^*,\omega^*)$ defined in \eqref{PsiLM} for $L,M$ large enough, and follows from \cite[Theorem 1]{bl20b}. The second step concerns the uniform boundedness of the inverse $D\Psi_{L,M}^{-1}(u^*,\alpha^*,\omega^*)$ and follows from \cite[Section 4.3]{bl20b} and \cite[Lemma 3.12]{ab20mas}, while keeping in mind the observations in \cite[Proposition 3.8]{ab22RE} concerning the RE components. The third and last step consists in proving that $\Psi_{L,M}(u^*,\alpha^*,\omega^*)$ vanishes and goes as the proof of \cite[Proposition 3.13]{ab20mas}.\end{proof}
\subsection{Final convergence results}
\label{s_convresult}
From the propositions in the previous subsections we can conclude that our problem of finding a fixed point of $\Phi$ in \eqref{Phi} satisfies all the assumptions required by \cite{mas15NM} under certain hypotheses on the state spaces, the discretization and the regularity of the right-hand sides. As a consequence, the relevant FEM converges.
\begin{theorem}[\protect{\cite[Theorem 2, page 539]{mas15NM}}] Under \ref{T1}-\ref{T6} and \ref{N1}-\ref{N8} there exists a positive integer $\hat{N}$ such that for all $L,M\geq\hat{N}$ the operator $R_{L}\Phi_{M}P_{L}$ has a fixed point $(u_{L,M}^{\ast},\alpha_{L,M}^{\ast},\omega_{L,M}^{\ast})$ and
\begin{equation*}
\setlength\arraycolsep{0.1em}\begin{array}{rcl}
\varepsilon_{L,M}&:=&\|(v_{L,M}^{\ast},\omega_{L,M}^{\ast})-(v^{\ast},\omega^{\ast})\|_{\mathbb{V}\times\mathbb{B}}\\[2mm]
&\leq&2\cdot\|[D\Psi_{L,M}(u^{\ast},\alpha^{\ast},\omega^{\ast})]^{-1}\|_{\mathbb{U}\times\mathbb{A}\times\mathbb{B}\leftarrow\mathbb{U}\times\mathbb{A}\times\mathbb{B}}\\[2mm]
&&\cdot\|\Psi_{L,M}(u^{\ast},\alpha^{\ast},\omega^{\ast})\|_{\mathbb{U}\times\mathbb{A}\times\mathbb{B}},
\end{array}
\end{equation*}
where $v_{L,M}^*=\mathcal{G}(u_{L,M}^*,\alpha_{L,M}^*)$ and $v^*=\mathcal{G}(u^*,\alpha^*)$.
\end{theorem}
Thanks to Proposition \ref{p_CS2b}, the error on $(v^{\ast},\omega^{\ast})$ is determined by the last factor, namely the \emph{consistency error}. For the latter, thanks to basic results on polynomial interpolation, we can write
\begin{equation*}
\|\Psi_{L,M}(u^{\ast},\alpha^{\ast},\omega^{\ast})\|_{\mathbb{U}\times\mathbb{A}\times\mathbb{B}}\leq\varepsilon_{L}+\max\{\Lambda_{m},\Lambda'_m,1\}\varepsilon_{M},
\end{equation*}
where $\Lambda_m$ is the Lebesgue constant associated to the collocation nodes and the terms
\begin{equation}\label{epsL}
\varepsilon_{L}:=\|(I_{\mathbb{U}\times\mathbb{A}\times\mathbb{B}}-P_{L}R_{L})(u^{\ast},\alpha^{\ast},\omega^{\ast})\|_{\mathbb{U}\times\mathbb{A}\times\mathbb{B}}
\end{equation}
and
\begin{equation}\label{epsM}
\varepsilon_{M}:=\|\Phi_{M}(u^{\ast},\alpha^{\ast},\omega^{\ast})-\Phi(u^{\ast},\alpha^{\ast},\omega^{\ast})\|_{\mathbb{U}\times\mathbb{A}\times\mathbb{B}}
\end{equation}
are called respectively {\it primary} and {\it secondary} consistency errors.

\bigskip
As for $\varepsilon_{L}$ in \eqref{epsL}, which concerns only the primary discretization, a bound can be obtained from the regularity of $u^{\ast}$ through the following theorem.
\begin{theorem}\label{t_epsL}
Let $K,H\in\mathcal{C}^{p}(\R\times\R^{2d},\mathbb{R}^{d})$ and $\tilde F,\tilde G\in\mathcal{C}^{p+1}(\R^d\times\mathtt{Y},\R^d)$ for some integer $p\geq0$. Then, Under \ref{T1}, \ref{T2}, \ref{N1} and \ref{N2}, it holds that $u^{\ast}\in C^{p+1}([0,1],\mathbb{R}^{2d})$, $\varphi^{\ast}\in C^{p+1}([-1,0],\mathbb{R}^{2d})$, $v^{\ast}\in C^{p+1}([-1,1],\mathbb{R}^{2d})$ and
\begin{equation}\label{epsLhco}
\varepsilon_{L}=O\left(h^{\min\{m,p+1\}}\right).
\end{equation}
\end{theorem}
\begin{proof}
Recall that $v^{\ast}=\mathcal{G}(u^{\ast},\alpha^{\ast})$ satisfies \eqref{rescaledBVPcoupled}, hence its periodic extension to $[-1,\infty)$ is a periodic solution of \eqref{eq:coupled} modulo rescaling of time, and it is  bounded by (T2). Thus, if $K,H$ are continuous and $\tilde F,\tilde G$ are continuously differentiable, so are $F,G$. Then the periodic extension of $v^{\ast}$ is continuously differentiable in $[0,\infty)$, thus also in $[-1,\infty)$ by periodicity. As a consequence, if $p=0$, $v^{\ast}\in C^1([-1,1],\mathbb{R}^d)$. Since $u^*=v^*\vert_{[0,1]}$ and $\alpha^*=v^*\vert_{[-1,1]}$, we immediately have also $u^{\ast}\in C^{1}([0,1],\mathbb{R}^{d})$ and $\alpha^*\in C^1([-1,0],\mathbb{R}^d)$. The whole reasoning can be iterated, proving the first part of the result.

To prove \eqref{epsLhco}, we observe first that 
\begin{equation}\label{epsLu1}
\|u^{\ast}-\pi_{L}^{+}\rho_{L}^{+}u^{\ast}\|_{\mathbb{U}}\leq\frac{\|{u^{\ast}}^{(m+1)}\|_{\infty}}{(m+1)!}\cdot h^{m+1}
\end{equation}
holds if $p\geq m$, while
\begin{equation}\label{epsLu2}
\|u^{\ast}-\pi_{L}^{+}\rho_{L}^{+}u^{\ast}\|_{\mathbb{U}}\leq(1+\Lambda_{m})\left(\frac{h}{2}\right)^{p+1}\frac{c_{p+1}}{m^{p+1}}\cdot\|{u^{\ast}}^{(p+1)}\|_{\infty}
\end{equation}
holds if $p\leq m$, with $c_{p+1}$ a positive constant independent of $m$. \eqref{epsLu1} is a direct consequence of the standard Cauchy interpolation reminder, see, e.g., \cite[Section 6.1, Theorem 2]{kc02}. \eqref{epsLu2} is a direct consequence of Jackson's theorem on best uniform approximation, see, e.g., \cite[(2.9) and (2.11)]{mas15I}. Recalling the definition of $\mathbb{A}$ in \ref{T2}, the same arguments hold for the first $d$ components of $\alpha^*$. However, recalling that $\|\psi\|_{B^{1,\infty}}=\|\psi\|_{\infty}+\|\psi'\|_{\infty}$, denoting by $\psi^*$ the last $d$ components of $\alpha^*$ and arguing as in \cite[Theorem 4.3]{ab20mas}, we get
$\|(\psi^{\ast}-\pi_{L}^{-}\rho_{L}^{-}\psi^{\ast})'\|_{\infty}=O\left(h^{\min\{m,p+1\}}\right)$, from which \eqref{epsLhco} follows.

\end{proof}
\bigskip
On the other hand, $\varepsilon_{M}$ in \eqref{epsM} concerns only the secondary discretization and is therefore absent whenever the latter is not needed. However, concerning our specific problem, according to \eqref{Phi} and \eqref{PhiMco}, it can be written as
\begin{equation}\label{epsMU}
\varepsilon_{M}:=\left\|\begin{pmatrix}
F_{M}(v^{\ast}_{\cdot}\circ s_{\omega^{\ast}})-F(v^{\ast}_{\cdot}\circ s_{\omega^{\ast}})\\
\omega^*(G_{M}(v^{\ast}_{\cdot}\circ s_{\omega^{\ast}})-G(v^{\ast}_{\cdot}\circ s_{\omega^{\ast}}))
\end{pmatrix}\right\|_{\mathbb{U}}
\end{equation}
and needs to be considered (at least) if the integrals in \eqref{hpcoupled2} cannot be exactly computed. If $\tilde F_M=\tilde F$ and $\tilde G_M=\tilde G$, \eqref{epsMU} is basically a quadrature error. Assuming that $M$ varies proportionally to $L$, one can choose a formula that guarantees at least the same order of the primary consistency error, so that the order of convergence of the final error is in fact the one given by theorem \ref{t_epsL}.

\begin{remark}\label{r_representation}
In principle, one could discretize the problem by choosing, for each mesh interval, a set of \emph{representation} nodes used to interpolate which are independent from the collocation nodes. That would mean that the unknowns of the discrete problem are given by the values of the relevant functions at the representation nodes, while the equations need to be satisfied at the collocation nodes. If $x^r_{L,M}$ is the vector of the unknowns and $Q_L:X_L\to X$ is the prolongation operator corresponding to the representation nodes (while $P_L,R_L$ refer to the collocation ones), the problem actually reads
$R_LQ_Lx^r_{L,M}=R_L\Phi_MQ_Lx^r_{L,M}.$
Thus, the vector $x^*_{L,M}$ given by the values of the relevant function at the collocation nodes is the solution of the discrete fixed point problem, in fact,
$$
x^*_{L,M}=R_LQ_Lx^r_{L,M}=R_L\Phi_MQ_Lx^r_{L,M}=R_L\Phi_MP_LR_LQ_Lx^r_{L,M}=R_L\Phi_MP_Lx^*_{L,M}.
$$
\end{remark}
%
\section{Numerical tests}\label{s_results}
This section deals with the numerical computation of periodic solutions of some specific equations from the field of population dynamics. {Let us remark that in our implementation (\url{http://cdlab.uniud.it/software}) we adopt Newton's method to solve the finite-dimensional system of nonlinear equations resulting from the fixed-point problem for $\Phi_{L,M}$ in \eqref{eq:philm}.}

\bigskip
\noindent The first system that we consider is
\begin{equation}\label{quadbvp}
\left\{\setlength\arraycolsep{0.1em}\begin{array}{rcll}
x(t) &=& \frac{\gamma}{2}\displaystyle\int_{-3}^{-1}x(t+\theta)(1-x(t+\theta))\dd\theta,& \\[1mm]
y'(t) &=& \gamma x(t)\displaystyle\int_{-3}^{-1}x(t+\theta)(1-x(t+\theta))\dd\theta+y(t),&
\end{array}
\right.
\end{equation}
for which the exact expression of the periodic solution between a Hopf bifurcation (at $\gamma=2+\pi/2$) and the first period doubling (at $\gamma\approx 4.327$) is 
\begin{equation}\label{quadbvpsol}
\left\{\setlength\arraycolsep{0.1em}\begin{array}{rcll}
x(t) &=& \sigma+A\sin\left(\frac{\pi t}{2}\right),& \\[1mm]
y(t) &=& -2\left(\sigma+2A\left(\frac{2\sin\left(\frac{\pi t}{2}\right)+\pi\cos\left(\frac{\pi t}{2}\right)}{4+\pi^2}\right)\right),&
\end{array}
\right.
\end{equation}
where $\sigma=1/2+\pi/(4\gamma)$ and $A^2= 2\sigma\left(1-1/\gamma-\sigma\right)$. \eqref{quadbvp} is artificially constructed from an RE studied in \cite{bdls16}, the same as the first of \eqref{quadbvp}, where the corresponding exact solution, the first of \eqref{quadbvpsol}, has been obtained. The RFDE component has been added to obtain again an exact expression for the solution.
The integral representing the distributed delay is approximated through a Clenshaw-Curtis quadrature \cite{tref00} rescaled to the interval $[-3,-1]$.

Starting from the exact solution at $\gamma=4$, the branch of periodic orbits is continued up to the first period doubling after the Hopf bifurcation.
The continuation is performed using a trivial phase condition by forcing $x(0)=\sigma$, and Chebyshev extrema as collocation points. The left panel of Figure \ref{fig:quadratic_errs} confirms the $O(h^{m})$ behavior (being $p=+\infty$) while showing that the error on the RE component converges with order $m+1$ instead. This is not surprising since the system is in fact uncoupled, the right-hand side of the RE component  {being} independent of the RFDE component, and $m+1$ is indeed the expected order of convergence for REs only \cite{ab22RE,adena21}.
\begin{figure}
\centering
\includegraphics[scale=0.75]{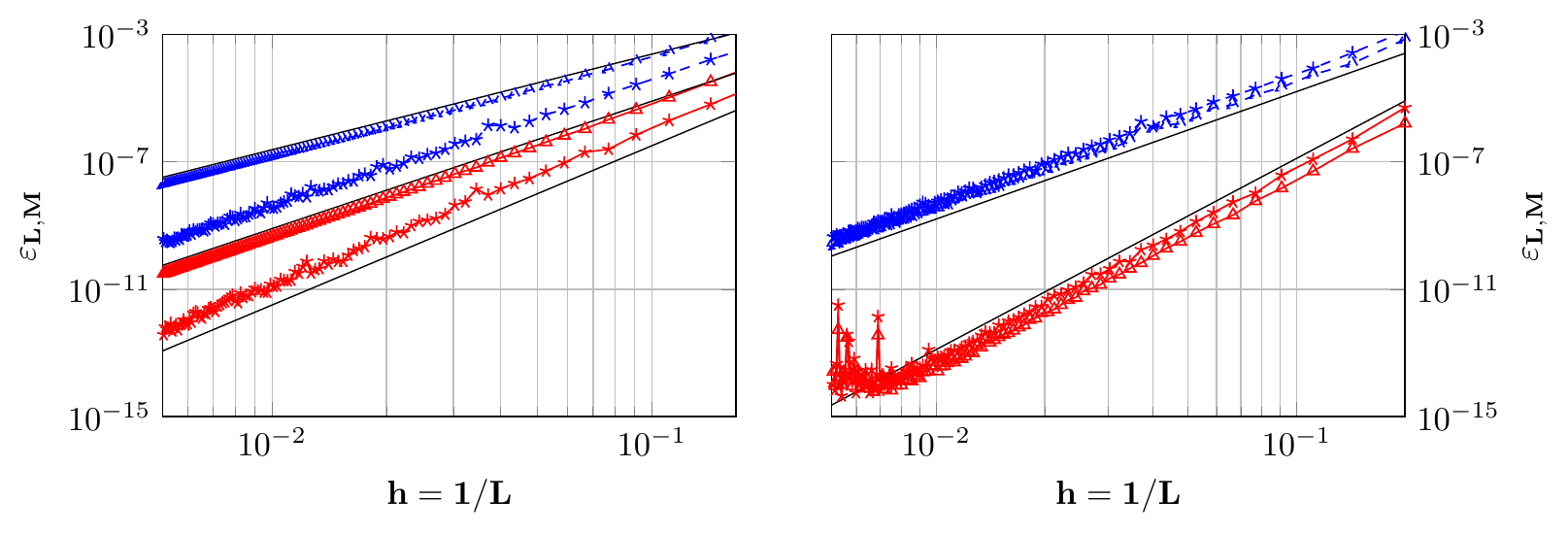}
\caption{Error on the periodic solution of \eqref{quadbvp} at $\gamma=4.327$. Left: $m=3$ (blue lines) and $m=4$ (red lines) using Chebyshev points, with triangles for the RFDE component and stars for the RE component, compared to straight lines having slopes 3, 4 and 5. Right: $m=3$ (blue lines) and $m=5$ (red lines) using Gauss-Legendre points compared to straight lines having slopes 4 and 6.}
\label{fig:quadratic_errs}
\end{figure}
\begin{figure}
\centering
\includegraphics[scale=0.75]{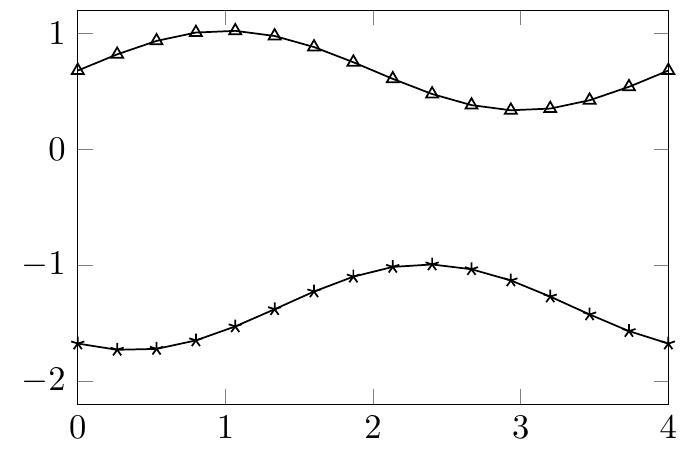}
\caption{Periodic solution of \eqref{quadbvp} at $\gamma=4.327$ using Gauss-Legendre points, $m=3$ and $L=5$, with triangles for the RFDE component and stars for the RE component.}
\label{fig:profile_exact}
\end{figure}
Given the experimental proof, found in \cite{elir00}, that in the case of RFDEs the order of convergence increases from $m$ to $m+1$ when using Gauss-Legendre collocation points, we replicate the above experiment using such collocation points, and the right panel of Figure \ref{fig:quadratic_errs} confirms the $O(h^{m+1})$ behavior for both components. Figure \ref{fig:profile_exact} shows an example of solution profile.

\bigskip
Next, consider the simplified logistic Daphnia model \cite{bdmv13}
\begin{equation}\label{simpledaphnia}
\left\{\setlength\arraycolsep{0.1em}\begin{array}{rcll}
b(t) &=& \beta S(t)\displaystyle\int_{\overline{a}}^{a_{\text{max}}}b(t-a)\dd a,& \\[3mm]
S'(t) &=& r \displaystyle S(t)\left(1-\frac{S(t)}{K}\right)-\gamma S(t)\int_{\overline{a}}^{a_{\text{max}}}b(t-a)\dd a.& 
\end{array}
\right.
\end{equation}
As shown in \cite{bdmv13}, for $r=K=\gamma=1$, $\overline{a}=3$ and $a_{\text{max}}=4$, a Hopf bifurcation occurs when $\beta \approx 3.0162$.

Starting from a periodic solution at $\beta=4$ computed using MatCont \cite{matcont} on the pseudospectral reduction to ODEs of \eqref{simpledaphnia} \cite{bdgsv16}, the branch of periodic orbits is continued up to $\beta=5$. Given the absence of an exact expression of the true solution, unlike the case \eqref{quadbvp}, the error is computed with respect to a reference solution which is in turn computed using $L=1000$ and $m=4$. Figure \ref{fig:prova_SIRS} confirms the $O(h^{m})$ behavior when using Chebyshev collocation nodes and the $O(h^{m+1})$ behavior when using the Gauss-Legendre nodes. Note that, unlike \eqref{quadbvp}, \eqref{simpledaphnia} is actually coupled, thus the order $m$ in the former case holds for both components. Figure \ref{fig:profile_daphnia} shows an example of solution profile.

\begin{figure}
\centering
\includegraphics[scale=0.75]{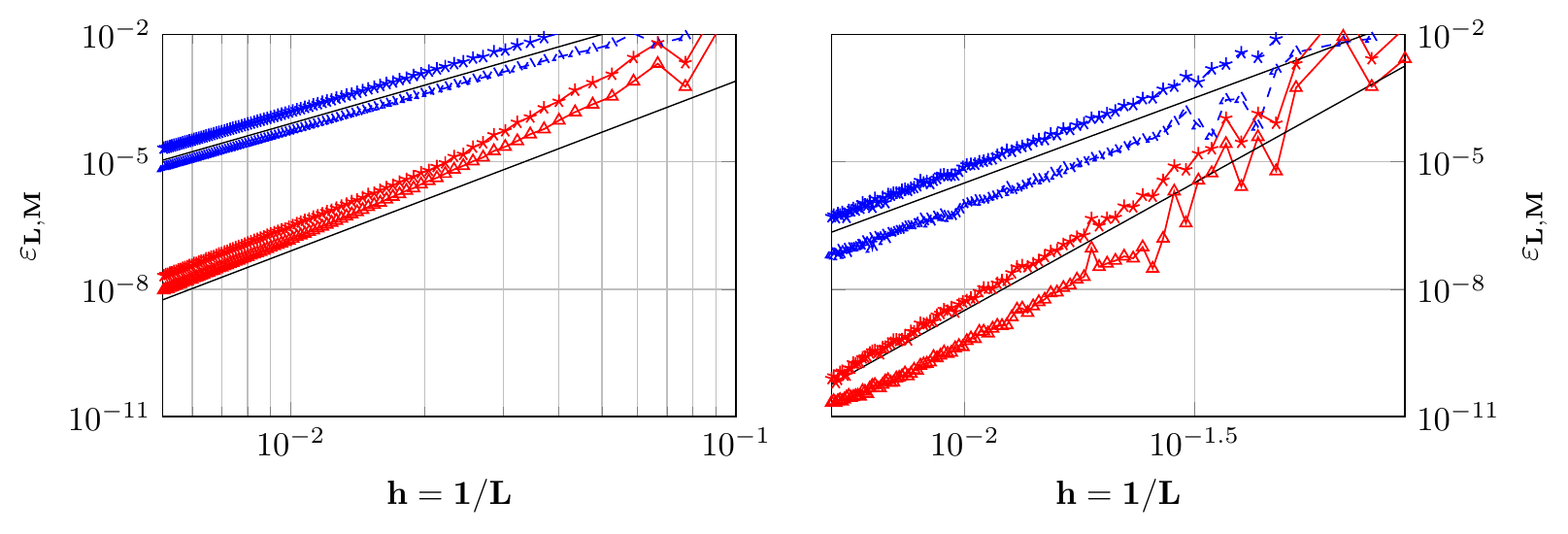}
\caption{Error on the periodic solution of \eqref{simpledaphnia} at $\beta=5$. Left: $m=3$ (blue lines) and $m=4$ (red lines) using Chebyshev points, with triangles for the RFDE component and stars for the RE component, compared to straight lines having slopes 3 and 4. Right: $m=3$ (blue lines) and $m=5$ (red lines) using Gauss-Legendre points compared to straight lines having slopes 4 and 6.}
\label{fig:prova_SIRS}
\end{figure}
\begin{figure}
\centering
\includegraphics[scale=0.75]{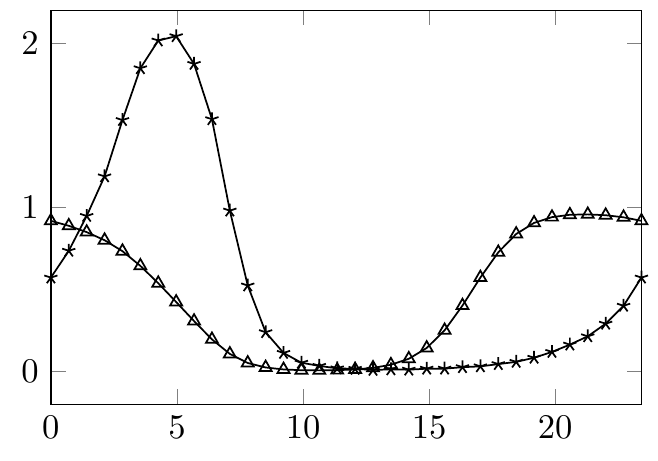}
\caption{Periodic solution of \eqref{simpledaphnia} at $\beta=5$ using Gauss-Legendre points, $m=3$ and $L=11$, with triangles for the RFDE component and stars for the RE component.}
\label{fig:profile_daphnia}
\end{figure}

\bigskip
Lastly, we consider the coupled system from the Plant neural model \cite{plant}
\begin{equation}\label{plant}
\left\{\setlength\arraycolsep{0.1em}\begin{array}{rcll}
v'(t) &=& v(t)-\displaystyle\frac{v(t)^3}{3}-w(t)+\mu(v(t-\tau)-v_0),& \\[1mm]
w(t) &=& w(t-\tau) +\displaystyle \int_{t-\tau}^{t}r(v(s)+a-bw(s))\dd s,& 
\end{array}
\right.
\end{equation}
where $v_0$ is a real root of $\displaystyle v-v^3/3-(v+a)/b$.
\eqref{plant} is derived by the original model in \cite{plant} by integrating the second RFDE. Observe that the second of \eqref{plant} is a \emph{neutral} RE, in that it does not satisfy \eqref{hpcoupled1}-\eqref{hpcoupled2} and is therefore not amenable of the analysis in this paper, but this does not impede to perform some numerical experiments to test its convergence behavior. Indeed, although the method seem to converge for several values of $L$ and $m$, the order of convergence is not as clear as in the previous test cases. Figure \ref{fig:profile_plant} shows an example of solution profile. Computing periodic solutions of neutral renewal equations (and relevant coupled systems) will be the subject of further study. As for the computation of relevant multipliers, see the recent work \cite{blvl23}.
\begin{figure}
\centering
\includegraphics[scale=0.75]{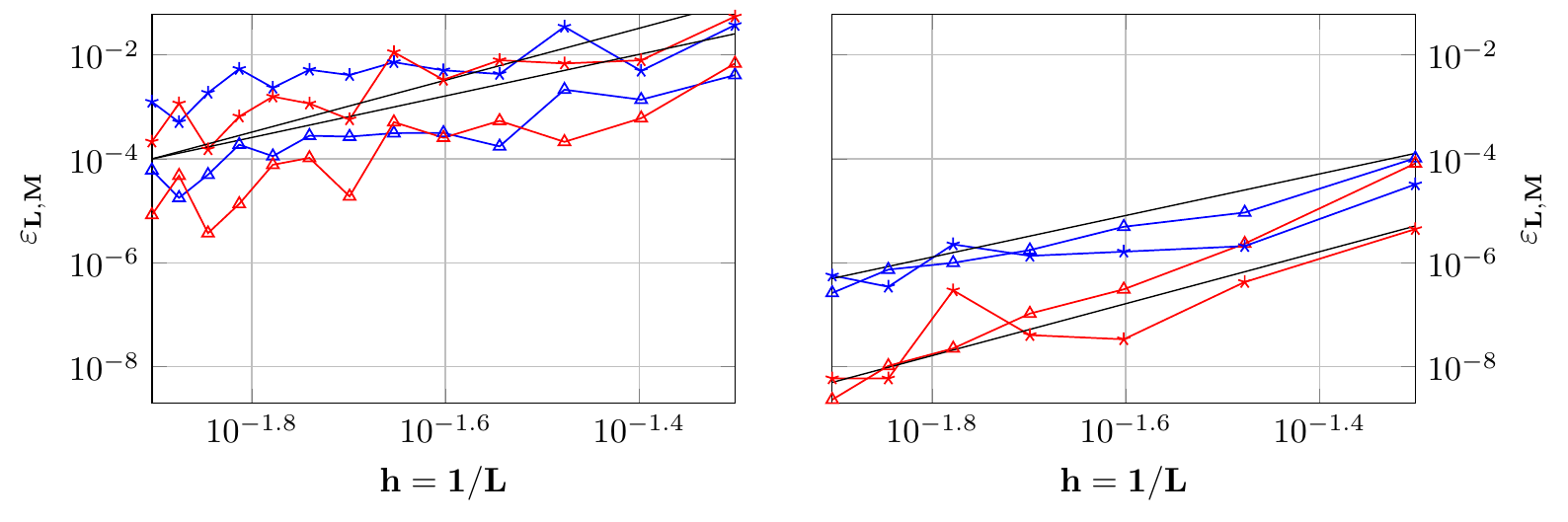}
\caption{Error on the periodic solution of \eqref{plant} using Gauss-Legendre points, using $m=4$ (blue lines) and $m=5$ (red lines), with triangles for the RFDE component and stars for the RE component, compared to straight lines having slopes 4 and 5. Left: $\tau=5$. Right: $\tau=2$.}
\label{fig:prova_C2}
\end{figure}
\begin{figure}
\centering
\includegraphics[scale=0.75]{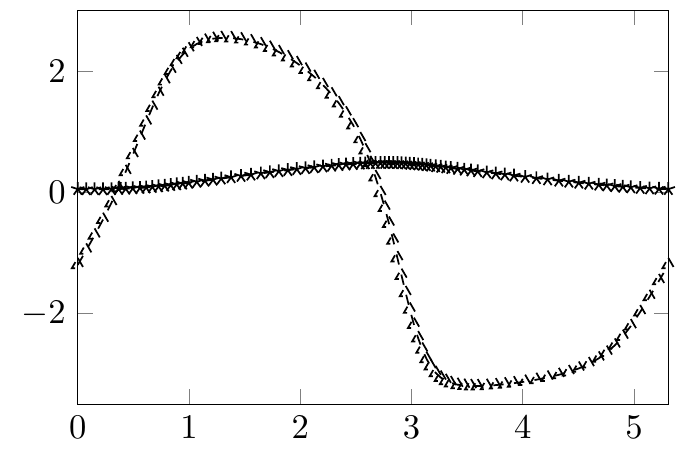}
\caption{Periodic solution of \eqref{plant} at $\tau=2$ using Gauss-Legendre points, $m=4$ and $L=20$, with triangles for the RFDE component and stars for the RE component.}
\label{fig:profile_plant}
\end{figure}

All the previous results concerned equations where $m<p+1$ (in fact, all the kernels were analytic). A case of an RE with $m>p$ has been already considered in \cite[(4.3)]{ab22RE}, see the relevant comments therein.

\section*{Acknowledgments}
The authors are members of INdAM Research group GNCS. Dimitri Breda is member of the UMI Research group “Modellistica socio-epidemiologica” and his work was partially supported by the Italian Ministry of University and Research (MUR) through the PRIN 2020 project (No. 2020JLWP23) “Integrated Mathematical Approaches to Socio-Epidemiological Dynamics” (CUP: E15F21005420006).
\appendix

\section{}
\label{s_appendix}

This appendix completes the proof of Proposition \ref{p_Ax*2co} by showing
that $\xi_1^*$ cannot be in $R$, following the notation introduced therein. To this aim, we first reformulate the linearization of both equations in \eqref{convBVPcoupled} as Volterra Integral Equations (VIEs)
with measure kernels {(\ref{reformulation})}, for the theory of which see \cite[Chapter 10]{grip09}. Then, we introduce the relevant adjoint equations {(\ref{adjeq})} and monodromy operators {(\ref{monop})}, showing that the latter are adjoint with respect to a suitable bilinear form. From this point on, the proof proceeds {similarly as for the cases of RFDEs and REs separately, dealt with, respectively, in \cite{ab20mas} and \cite{ab22RE}}.

\subsection{{Reformulation as VIE with measure kernel}}\label{reformulation}
Under the assumptions $\omega\geq\tau$, \eqref{hpcoupled1} and \eqref{hpcoupled2}, the linearization of the equations in \eqref{convBVPcoupled} around the periodic solution $(x^*, y^*, \omega^*)$ reads
\begin{equation}\label{BVPlinearized}
\left\{\setlength\arraycolsep{0.1em}\begin{array}{rcll}
x(t) &=& \mathfrak{L}^*_X(t)[(x_t, y_t)\circ s_{\hat\omega}]+\omega\mathfrak{M}^*_X(t),&\quad t\in[0,1], \\[1mm]
y'(t) &=& \mathfrak{L}^*_Y(t)[(x_t, y_t)\circ s_{\hat\omega}]+\omega\mathfrak{M}^*_Y(t),&\quad t\in[0,1], \\[1mm]
\end{array}
\right.
\end{equation}
where $\mathfrak{L}_X$, $\mathfrak{M}_X$, $\mathfrak{L}_Y$ and $\mathfrak{M}_Y$ have been defined in \eqref{Lco1}, \eqref{Mco1}, \eqref{Lco2} and \eqref{Mco2}, respectively.
Due to the second of \eqref{hpcoupled2}, since $D_2\tilde G$ is continuous, the homogeneous version of the RFDE in \eqref{BVPlinearized} can be written in the form
\begin{equation}\label{RFDElinearized}
y'(t)=\int_{-r}^0h(t,\theta)(x(t+\theta),y(t+\theta))\dd\theta+\int_{-r}^0\dd_{\theta}\eta(t,\theta)y(t+\theta),
\end{equation}
where $r:=\tau/\omega\leq 1$, $\eta(t,\cdot):[-r,0]\to\R^{d\times d}$ is NBV for all $t\in\R$, $h(t,\cdot):[-r,0]\to\R^{d\times d}$ is measurable for all $t\in\R$,  $h(\cdot,\sigma)$ and $\eta(\cdot,\sigma)$ are $1$-periodic for almost all $\sigma\in[-r,0]$.
Integrating \eqref{RFDElinearized} one gets
\begin{equation*}
\setlength\arraycolsep{0.1em}\begin{array}{rcl}
y(t)&=y(t_0)+&\displaystyle\int_{t_0}^t \left[\int_{s-r}^s h(s,\theta-s)(x(\theta),y(\theta))\dd\theta+\int_{s-r}^s\dd_{\theta}\eta(s,\theta-s)y(\theta)\right]\dd s.
\end{array}
\end{equation*}
By changing the order of integration, we get the VIE
\begin{equation}\label{VIEY}
y(t)=\displaystyle\int_{t_0}^tC(t,\theta-t,\dd\theta)(x(\theta), y(\theta))+g(t),
\end{equation}
with $$
g(t):=\displaystyle\int_{t-1}^{t_0}C(t,\theta-t,\dd\theta)(x(\theta), y(\theta))+y(t_0),
$$
where
$$
C(t,\theta-t,\dd\theta)(x(\theta), y(\theta)):=\begin{cases}\displaystyle\int_{\max\{t_0-t,\theta-t\}}^{0}[h(s+t,\theta-t-s)(x(\theta),y(\theta))\dd\theta\\
\qquad\qquad+\dd_{\theta}\eta(s+t,\theta-t-s)y(\theta)]\dd s,\\
\qquad\qquad \qquad t\geq t_0, \theta\in[t-r,t],\\
0,\qquad\qquad\qquad\textrm{otherwise}.
\end{cases}
$$
Similarly to the RFDE, the homogeneous version of the RE in \eqref{BVPlinearized} can be expressed through the Riemann-Stjeltis integral
$$
x(t)=\int_{-r}^0k(t,\theta)(x(t+\theta),y(t+\theta))\dd\theta+\int_{-r}^0\dd_{\theta}\chi(t,\theta)y(t+\theta)
$$
where $\chi(t,\cdot):[-r,0]\to\R^{d\times d}$ is NBV for all $t\in\R$, $k(t,\cdot):[-r,0]\to\R^{d\times d}$ is measurable for all $t\in\R$, $k(\cdot,\sigma)$ and $\chi(\cdot,\sigma):\R\to\R^{d\times d}$ are $1$-periodic for all $\sigma\in[-r,0]$.

Thus, it can in turn be expressed a VIE with measure kernel (see \cite[Chapter 10]{grip09}) as
\begin{equation}\label{VIEX}
x(t)=\displaystyle\int_{t_0}^tK(t,\theta-t,\dd\theta)(x(\theta), y(\theta))+f(t),
\end{equation}
where 
$$
K(t,\theta-t,\dd\theta)(x(\theta), y(\theta)):=\begin{cases}
k(t,\theta-t)(x(\theta),y(\theta))\dd\theta\\
\,\qquad +\dd_{\theta}\chi(t,\theta-t)y(\theta),&t\geq t_0, \theta\in[t-r,t],\\
0,&\textrm{otherwise},
\end{cases}
$$
and
$$
f(t):=\displaystyle\int_{t-1}^{t_0}K(t,\theta-t,\dd\theta)(x(\theta), y(\theta)).
$$

Eventually, both \eqref{VIEY} and \eqref{VIEX} can be written as a single integral equation with measure kernel defined in $\R^{2 d}$, viz.
\begin{equation}\label{VIEcoupled}
z(t)=\displaystyle\int_{t_0}^tW(t,\theta-t,\dd\theta)z(\theta)+j(t).
\end{equation}
Consider the VIE with measure kernel \eqref{VIEcoupled}. By \cite[Chapter 10, Corollary 2.9]{grip09}, if $I:=\R$ can be finitely partitioned into subintervals $\{I_i\}_{i\leq n}$ such that
\begin{equation}\label{pwleq1}
\sup_{t\in I_i}|W|(t,[-r,0],I_i)\leq 1,\qquad i\leq n,
\end{equation}
then $W$ has a resolvent of type $B^{\infty}$ on $I$, i.e., there exists a measure kernel $R$ of type $B^{\infty}$ such that
\begin{equation*}
\setlength\arraycolsep{0.1em}\begin{array}{rcl}
W(s,[-r,0],E)&=&\displaystyle R(s,[-r,0],E)+\int_{t_0}^{t}R(s,\sigma-s,\dd\sigma)W(\sigma,[-r,0],E)\\[4mm]
&=&\displaystyle R(s,[-r,0],E)+\int_{t_0}^{t}W(s,\sigma-s,\dd\sigma)R(\sigma,[-r,0],E),\quad s\in I,
\end{array}
\end{equation*}
for $E\subset I$ Borel. Note that \eqref{pwleq1} holds if, for instance, $W$ is of bounded total variation. Indeed this implies, by definition, that there is at most a finite number of points $x_0,\ldots,x_n$ such that any subinterval $I'\subset I$ containing one of such points satisfies $\sup_{t\in I'}|W|(t,[-r,0],I')> 1$. Consider the partition of $I$ given by $(-\infty,x_0]$, $[x_i,x_{i+1}]$ for $i=0,\ldots,n$ and $[x_n,+\infty)$. By construction, such partition  can be further refined in order to satisfy \eqref{pwleq1}.
 
\subsection{{Adjoint equations}}\label{adjeq}
Define the adjoint VIE with measure kernel \cite[Chapter 10]{grip09} 
\begin{equation}\label{VIEadjoint}
\zeta(s)=\int_{s}^{s_{0}}\zeta(\sigma)W(s,\sigma-s,\dd\sigma)+\gamma(s),\quad s\leq s_{0},
\end{equation}
with
\begin{equation}\label{gamma_s}
\gamma(s):=\int_{s_{0}}^{s+1}\psi_{0}(\sigma)W(s,\sigma-s,\dd\sigma),\quad s\leq s_{0}
\end{equation}
for
\begin{equation*}
\psi_{0}(s_{0}+\eta):=\begin{cases}
\psi(\eta),&\eta\in[0,1],\\
0,&\textrm{otherwise},
\end{cases}
\end{equation*}
and $\zeta_{s_{0}}=\psi$ for some $\psi\in Z^*:=L^{\infty}([0,1],\R^d)\times Y^*$.
Existence and uniqueness \cite[Chapter 10, Theorem 2.5]{grip09} allows to define the forward evolution family $\{T(t,t_{0})\}_{t\geq t_{0}}$ on $Z:=X\times Y$ through $T(t,t_{0})z_{t_0}=z_{t}$ and that of its adjoint $\{V(s,s_0)\}_{s\leq s_0}$ on $Z^*$ through $V(s,s_0)\zeta_{s_0}=\zeta_s$. Moreover, we can express the solution of \eqref{VIEcoupled} and that of its adjoint \eqref{VIEadjoint} respectively as
\begin{equation*}
z(t)=j(t)+\int_{t_{0}}^{t}R(t,\sigma-t,\dd\sigma)j(\sigma)\quad t\geq t_{0},
\end{equation*}
and
\begin{equation*}
\zeta(s)=\gamma(s)+\int_{s}^{s_{0}}\gamma(\sigma)R(\sigma,s-\sigma,\dd\sigma),\quad s\leq s_{0}.
\end{equation*}

Given $t\in\mathbb{R}$, consider the pairing $[\cdot,\cdot]_{t}:Z^*\times Z\to\mathbb{R}$ defined as
\begin{equation}\label{bilinear2}
[\psi,\varphi]_{t}:=\int_{0}^{1}\psi(\eta)\int_{-1}^{0}W(t+\eta,\beta-\eta,\dd\beta)\varphi(\beta)\dd\eta.
\end{equation}

Observe that such bilinear form is nondegenerate for all $t\in\R$ whenever $W$ is nontrivial. Indeed, assume by contradiction that there exists $\psi\in Z^*$ such that $\psi$ is nonzero but $[\psi,\cdot]_t$ is constantly 0. By the nondegenerateness of the standard bilinear form, this means that the innermost integral is 0 for all $\varphi\in Z$ and almost all $\eta\in[0,1]$. If $\varphi:=z_t$, where $z$ is the (unique modulo multiplication by constant) periodic solution of the VIE, then such integral is equal to
\begin{equation*}
\int_{-1}^0W(t+\eta,\beta-\eta,\dd\beta)z(t+\beta)=\int_{t-1}^tW(t+\eta,\beta-t-\eta,\dd\beta)z(\beta)=z(t+\eta).
\end{equation*}
Thus $z_{t+1}$ is almost everywhere equal to 0. Using periodicity, this means that $z$ is almost everywhere 0, which is only possible if $W$ is trivial, contradiction. Using similar arguments one can prove that there is no nonzero $\varphi\in Z$ such that $[\cdot,\varphi]_t$ is constantly zero, after exchanging the order of integration in the definition of $[\cdot,\cdot]_t$.

\subsection{{Adjoint monodromy operators}}\label{monop}
We claim that the forward monodromy operator and the corresponding
backward one are adjoint w.r.t. \eqref{bilinear2}, i.e., that
\begin{equation}\label{VT2}
[V(t-1,t)\psi,\varphi]_{t}=[\psi,T(t+1,t)\varphi]_{t}.
\end{equation}

We have
\begin{equation*}
\setlength\arraycolsep{0.1em}\begin{array}{rcl}
[V(t-1,t)\psi,\varphi]_{t}&=&\displaystyle\int_{0}^{1}[V(t-1,t)\psi](\eta)\int_{-1}^{0}W(t+\eta,\beta-\eta,\dd\beta)\varphi(\beta)\dd\eta\\[4mm]
&=&\displaystyle\int_{0}^{1}\zeta(t-1+\eta;t,\psi)\int_{-1}^{0}W(t+\eta,\beta-\eta,\dd\beta)\varphi(\beta)\dd\eta\\[4mm]

&=&\displaystyle\int_{0}^{1}\gamma(t-1+\eta)\int_{-1}^{0}W(t+\eta,\beta-\eta,\dd\beta)\varphi(\beta)\dd\eta\\[4mm]
&+&\displaystyle\int_{0}^{1}\int_{t-1+\eta}^{t}\gamma(\sigma)R(\sigma,t-1+\eta-\sigma,\dd\sigma)\\[4mm]
&&\displaystyle\int_{-1}^{0}W(t+\eta,\beta-\eta,\dd\beta)\varphi(\beta)\dd\eta
\\[4mm]
&=&A+B
\end{array}
\end{equation*}
for
\begin{equation*}
A:=\displaystyle\int_{0}^{1}\gamma(t-1+\eta)\int_{-1}^{0}W(t+\eta,\beta-\eta,\dd\beta)\varphi(\beta)\dd\eta
\end{equation*}
and
\begin{equation*}
B:=\displaystyle\int_{0}^{1}\int_{t-1+\eta}^{t}\gamma(\sigma)R(\sigma,t-1+\eta-\sigma,\dd\sigma)\int_{-1}^{0}W(t+\eta,\beta-\eta,\dd\beta)\varphi(\beta)\dd\eta.
\end{equation*}

As for $A$, we have
\begin{equation*}
\setlength\arraycolsep{0.1em}\begin{array}{rcl}
A&=&\displaystyle\int_{0}^{1}\int_{t}^{t+\eta}\psi_{0}(\sigma)W(\sigma,t-1+\eta-\sigma,\dd\sigma)\int_{-1}^{0}W(t+\eta,\beta-\eta,\dd\beta)\varphi(\beta)\dd\eta\\[4mm]
&=&\displaystyle\int_{0}^{1}\int_{0}^{\eta}\psi(\sigma)W(\sigma+t,\eta-1-\sigma,\dd\sigma)\int_{-1}^{0}W(t+\eta,\beta-\eta,\dd\beta)\varphi(\beta)\dd\eta\\[4mm]

&=&\displaystyle\int_{0}^{1}\int_{0}^{1}\psi(\sigma)W(\sigma+t,\eta-1-\sigma,\dd\sigma)\int_{-1}^{0}W(t+\eta,\beta-\eta,\dd\beta)\varphi(\beta)\dd\eta\\[4mm]

&=&\displaystyle\int_{0}^{1}\psi(\sigma)\int_{0}^{1}W(\sigma+t,\eta-1-\sigma,\dd\sigma)\int_{-1}^{0}W(t+\eta,\beta-\eta,\dd\beta)\varphi(\beta)\dd\eta\\[4mm]
&=&\displaystyle\int_{0}^{1}\psi(\sigma)\int_{-1}^{0}\int_{0}^{1}W(\sigma+t,\eta-1-\sigma,\dd\sigma)W(t+\eta,\beta-\eta,\dd\beta)\varphi(\beta)\dd\eta\\[4mm]
&=&\displaystyle\int_{0}^{1}\psi(\sigma)\int_{-1}^{0}\int_{0}^{1}W(\sigma+t,\eta-1-\sigma,\dd\sigma)W(t+\eta,\beta-\eta,\dd\beta)\varphi(\beta)\dd\eta\\[4mm]
&=&\displaystyle\int_{0}^{1}\psi(\sigma)\int_{-1}^{0}\int_{0}^{1}W(\sigma+t,\eta-1-\sigma,\dd\sigma)W(t+\eta,\beta-\eta,\dd\beta)\varphi(\beta)\dd\eta\\[4mm]
&=&\displaystyle\int_{0}^{1}\psi(\sigma)\int_{-1}^{0}\int_{-1}^{0}W(\sigma+t,\eta-\sigma,\dd\sigma)W(t+\eta+1,\beta-\eta-1,\dd\beta)\varphi(\beta)\dd\eta\\[4mm]
&=&\displaystyle\int_{0}^{1}\psi(\sigma)\int_{-1}^{0}\int_{-1}^{\beta}W(\sigma+t,\eta-\sigma,\dd\sigma)W(t+\eta+1,\beta-\eta-1,\dd\beta)\varphi(\beta)\dd\eta\\[4mm]
&=&\displaystyle\int_{0}^{1}\psi(\sigma)\int_{-1}^{0}W(\sigma+t,\eta-\sigma,\dd\sigma)\int_{\eta}^{0}W(t+\eta+1,\beta-\eta-1,\dd\beta)\varphi(\beta)\dd\eta\\[4mm]
&=&\displaystyle\int_{0}^{1}\psi(\sigma)\int_{-1}^{0}W(\sigma+t,\eta-\sigma,\dd\sigma)\\[4mm]
&&\displaystyle\int_{t+\eta}^{t}W(t+\eta+1,\beta-\eta-t-1,\dd\beta)\varphi_0(\beta)\dd\eta\\[4mm]
&=&\displaystyle\int_{0}^{1}\psi(\sigma)\int_{-1}^{0}W(\sigma+t,\eta-\sigma,\dd\sigma)j(t+1+\eta)\dd\eta\\[4mm]
\end{array}
\end{equation*}
where the first equality comes from \eqref{gamma_s} by setting $s_0=t$ and $s=t-1+\eta$, the second equality follows from the substitution $\sigma\leftarrow t+\sigma$, the third from the fact that $W(\sigma+t,\eta-1-\sigma,\dd\sigma)$ vanishes for $\eta<\sigma$, the fourth by exchanging the order of integration between $\eta$ and $\sigma$, the fifth by exchanging the order of integration between $\beta$ and $\sigma$, the seventh from the substitution $\eta\leftarrow\eta+1$, the eight from the fact that $W(t+\eta+1,\beta-\eta-1,\dd\beta)$ vanishes for $\eta>\beta$, the ninth by exchanging the order of integration between $\beta$ and $\eta$, the tenth from the substitution $\beta\leftarrow\beta-t$ and the last one from the definition of $j$ by setting $t_0=t$.
As for $B$, we have
\begin{equation*}
\setlength\arraycolsep{0.1em}\begin{array}{rcl}
B&=&\displaystyle\int_{0}^{1}\int_{t-1+\eta}^{t}\gamma(\sigma)R(\sigma,t-1+\eta-\sigma,\dd\sigma)\int_{-1}^{0}W(t+\eta,\beta-\eta,\dd\beta)\varphi(\beta)\dd\eta\\[4mm]
&=&\displaystyle\int_{0}^{1}\int_{t-1+\eta}^{t}\int_t^{\sigma+1}\psi_0(\theta)W(\theta,\sigma-\theta,\dd\theta)R(\sigma,t-1+\eta-\sigma,\dd\sigma)\\[4mm]
&&\displaystyle\int_{-1}^{0}W(t+\eta,\beta-\eta,\dd\beta)\varphi(\beta)\dd\eta\\[4mm]

&=&\displaystyle\int_{0}^{1}\int_{t-1+\eta}^{t}\int_0^{\sigma+1-t}\psi(\theta)W(t+\theta,\sigma-t-\theta,\dd\theta)R(\sigma,t-1+\eta-\sigma,\dd\sigma)\\[4mm]
&&\displaystyle\int_{-1}^{0}W(t+\eta,\beta-\eta,\dd\beta)\varphi(\beta)\dd\eta\\[4mm]
&=&\displaystyle\int_{0}^{1}\int_{0}^{1}\psi(\theta)\int_{t-1+\max\{\eta,\theta\}}^{t}W(t+\theta,\sigma-t-\theta,\dd\theta)R(\sigma,t-1+\eta-\sigma,\dd\sigma)\\[4mm]
&&\displaystyle\int_{-1}^{0}W(t+\eta,\beta-\eta,\dd\beta)\varphi(\beta)\dd\eta\\[4mm]

&=&\displaystyle\int_{0}^{1}\int_{0}^{1}\psi(\theta)\int_{t-1}^{t}W(t+\theta,\sigma-t-\theta,\dd\theta)R(\sigma,t-1+\eta-\sigma,\dd\sigma)\\[4mm]
&&\displaystyle\int_{-1}^{0}W(t+\eta,\beta-\eta,\dd\beta)\varphi(\beta)\dd\eta\\[4mm]

&=&\displaystyle\int_{0}^{1}\psi(\theta)\int_{t-1}^{t}W(t+\theta,\sigma-t-\theta,\dd\theta)\int_{0}^{1}R(\sigma,t-1+\eta-\sigma,\dd\sigma)\\[4mm]
&&\displaystyle\int_{-1}^{0}W(t+\eta,\beta-\eta,\dd\beta)\varphi(\beta)\dd\eta\\[4mm]
&=&\displaystyle\int_{0}^{1}\psi(\theta)\int_{t-1}^{t}W(t+\theta,\sigma-t-\theta,\dd\theta)\int_{0}^{1}R(\sigma+1,t-1+\eta-\sigma,\dd\sigma)\\[4mm]

&&\displaystyle\int_{-1}^{0}W(t+\eta,\beta-\eta,\dd\beta)\varphi(\beta)\dd\eta\\[4mm]
&=&\displaystyle\int_{0}^{1}\psi(\theta)\int_{t-1}^{t}W(t+\theta,\sigma-t-\theta,\dd\theta)\int_{t}^{t+1}R(\sigma+1,\eta-1-\sigma,\dd\sigma)\\[4mm]
&&\displaystyle\int_{-1}^{0}W(\eta,t+\beta-\eta,\dd\beta)\varphi(\beta)\dd\eta\\[4mm]

&=&\displaystyle\int_{0}^{1}\psi(\theta)\int_{t-1}^{t}W(t+\theta,\sigma-t-\theta,\dd\theta)\int_{t}^{\sigma+1}R(\sigma+1,\eta-1-\sigma,\dd\sigma)\\[4mm]
&&\displaystyle\int_{\eta-t-1}^{0}W(\eta,t+\beta-\eta,\dd\beta)\varphi(\beta)\dd\eta\\[4mm]
\end{array}
\end{equation*}
\begin{equation*}
\setlength\arraycolsep{0.1em}\begin{array}{rcl}

&=&\displaystyle\int_{0}^{1}\psi(\theta)\int_{-1}^{0}W(t+\theta,\sigma-\theta,\dd\theta)\int_{t}^{\sigma+t+1}R(\sigma+t+1,\eta-1-\sigma-t,\dd\sigma)\\[4mm]
&&\displaystyle\int_{\eta-1}^{t}W(\eta,\beta-\eta,\dd\beta)\varphi_0(\beta)\dd\eta\\[4mm]

&=&\displaystyle\int_{0}^{1}\psi(\theta)\int_{-1}^{0}W(t+\theta,\sigma-\theta,\dd\theta)\\[4mm]
&&\displaystyle\int_{t}^{\sigma+t+1}R(\sigma+t+1,\eta-1-\sigma-t,\dd\sigma)j(\eta)\dd\eta\\[4mm]

&=&\displaystyle\int_{0}^{1}\psi(\sigma)\int_{-1}^{0}W(t+\sigma,\eta-\sigma,\dd\sigma)\\[4mm]
&&\displaystyle\int_{t}^{\eta+t+1}R(\eta+t+1,\theta-1-\eta-t,\dd\eta)j(\theta)\dd\theta,
\end{array}
\end{equation*}
where the first equality comes from \eqref{gamma_s}, the second from the substitution $\theta\leftarrow t+\theta$, the third by exchanging the order of integration between $\theta$ and $\sigma$, the fourth from the fact that $R(\sigma,t-1+\eta-\sigma,\dd\sigma)$ vanishes for $\sigma<t-1+\eta$ and $W(t+\theta,\sigma-t-\theta,\dd\theta)$ vanishes for $\sigma<t-1+\theta$, the fifth by permuting the order of integration of $\eta$, $\sigma$ and $\theta$, the sixth from the $1$-periodicity of $R$ with respect to its first argument, the seventh from the substitution $\eta\leftarrow\eta-t$, the eighth from the fact that $R(\sigma+1,t-1+\eta-\sigma,\dd\sigma)$ vanishes for $\eta>1+\sigma$ and $W(t+\theta,\sigma-t-\theta,\dd\theta)$ vanishes for $\beta<\eta-t-1$, the ninth from the substitutions $\sigma\leftarrow\sigma+t$ and $\beta\leftarrow\beta-t$, the tenth from the definition of $j$ and the last one by renaming the variables.
Eventually, we get
\begin{equation*}
\setlength\arraycolsep{0.1em}\begin{array}{rcl}
A+B&=&\displaystyle\int_{0}^{1}\psi(\sigma)\int_{-1}^{0}W(\sigma+t,\eta-\sigma,\dd\sigma)j(t+1+\eta)\dd\eta\\[4mm]
+&&\displaystyle\int_{0}^{1}\psi(\sigma)\displaystyle\int_{-1}^{0}W(t+\sigma,\eta-\sigma,\dd\sigma)\\[4mm]
&&\displaystyle\int_{t}^{\eta+t+1}R(\eta+t+1,\theta-1-\eta-t,\dd\eta)j(\theta)\dd\theta\\[4mm]
&=&\displaystyle\int_{0}^{1}\psi(\sigma)\displaystyle\int_{-1}^{0}W(\sigma+t,\eta-\sigma,\dd\sigma)z(t+1+\eta;t,\varphi)\dd\eta\\[4mm]
&=&\displaystyle\int_{0}^{1}\psi(\sigma)\displaystyle\int_{-1}^{0}W(\sigma+t,\eta-\sigma,\dd\sigma)[T(t+1,t)\varphi](\eta)\dd\eta\\[4mm]
&=&[\psi,T(t+1,t)\varphi]_t,
\end{array}
\end{equation*}
which proves \eqref{VT2}. As anticipated, the proof is then concluded by following the last steps of{ the corresponding proofs in \cite{ab20mas,ab22RE}}.

\bibliographystyle{abbrv}


\end{document}